\documentclass[12pt]{amsart}
\usepackage{a4wide,enumerate,xcolor,graphicx}
\usepackage{amsmath,bm}
\allowdisplaybreaks

\let\pa\partial
\let\na\nabla
\let\eps\varepsilon
\newcommand{\N}{{\mathbb N}}
\newcommand{\R}{{\mathbb R}}

\newcommand{\diver}{\operatorname{div}}
\newcommand{\ran}{\operatorname{ran}}
\renewcommand{\H}{{\mathcal H}}
\newcommand{\E}{{\mathcal E}}
\newcommand{\evmin}{\lambda_m}
\newcommand{\evmax}{\lambda_M}

\newtheorem{theorem}{Theorem}
\newtheorem{lemma}[theorem]{Lemma}

\begin{document}

\title[Maxwell--Stefan--Cahn--Hilliard systems]{Existence and weak-strong uniqueness
for Maxwell--Stefan--Cahn--Hilliard systems} 

\author[X. Huo]{Xiaokai Huo}
\address{Institute of Analysis and Scientific Computing, Technische Universit\"at 
Wien, Wiedner Hauptstra\ss e 8--10, 1040 Wien, Austria}
\email{xiaokai.huo@tuwien.ac.at} 

\author[A. J\"ungel]{Ansgar J\"ungel}
\address{Institute of Analysis and Scientific Computing, Technische Universit\"at 
Wien, Wiedner Hauptstra\ss e 8--10, 1040 Wien, Austria}
\email{juengel@tuwien.ac.at} 

\author[A. Tzavaras]{Athanasios E. Tzavaras}
\address{Computer, Electrical and Mathematical Science and Engineering Division,
King Abdullah University of Science and Technology (KAUST),
Thuwal 23955-6900, Saudi Arabia}
\email{athanasios.tzavaras@kaust.edu.sa}
\date{\today}

\thanks{XH and AJ acknowledge partial support from   
the Austrian Science Fund (FWF), grants P33010, W1245, and F65.
AET acknowledges support from the King Abdullah University of Science and Technology (KAUST).
This work has received funding from the European 
Research Council (ERC) under the European Union's Horizon 2020 research and 
innovation programme, ERC Advanced Grant no.~101018153.} 

\begin{abstract}
A Maxwell--Stefan system for fluid mixtures with driving forces depending on
Cahn--Hilliard-type chemical potentials is analyzed. The corresponding parabolic
cross-diffusion equations contain fourth-order derivatives and are considered
in a bounded domain with no-flux boundary conditions. The main difficulty of
the analysis is the degeneracy of the diffusion matrix, which is overcome by
proving the positive definiteness of the matrix on a subspace and using the
Bott--Duffin matrix inverse. The global existence of weak solutions and a 
weak-strong uniqueness property are shown by a careful combination of 
(relative) energy and entropy estimates, yielding $H^2(\Omega)$ bounds for
the densities, which cannot be obtained from the energy or entropy inequalities alone.
\end{abstract}

% \paragraph{Keywords:}  
\keywords{Cross-diffusion systems, global existence, weak-strong uniqueness, 
relative entropy, relative free energy, parabolic fourth-order equations, 
Maxwell--Stefan equations, Cahn--Hilliard equations.}  
 
% \paragraph{AMS classification:}  
\subjclass[2000]{35A02, 35G20, 35G31, 35K51, 35K55, 35Q35.}

\maketitle

%%%%%%%%%%%%%%%%%%%%%%%%%%%%%%%%%%%%%%%%%%%%%%%%%%%%%%%%%%%%%%%%%%%%%%%%%%%%%%%

\section{Introduction}

The evolution of fluid mixtures is important in many scientific fields like
biology and nanotechnology to understand the 
diffusion-driven transport of the species. The transport can be modeled 
by the Maxwell--Stefan equations \cite{Max66,Ste71}, which
consist of the mass balance equations and the relations between the driving forces
and the fluxes. The driving forces involve the chemical potentials of the
species, which in turn are determined by the (free) energy.
When the fluid is immiscible, the energy can be assumed to consist 
of the thermodynamic entropy and the phase separation energy, 
given by a density gradient \cite{CaHi58}. The gradient energetically penalizes
the formation of an interface and restrains the segregation.
This leads to a system of cross-diffusion equations with fourth-order derivatives. 
The aim of this paper is to provide a
global existence and weak-strong uniqueness analysis for the multicomponent
Maxwell--Stefan--Cahn--Hilliard system.

\subsection{Model equations and state of the art}

The equations for the partial densities $c_i$ and partial velocities $u_i$
are given by
\begin{align}
  \pa_t c_i + \diver(c_iu_i) &= 0,   \quad i = 1,\ldots, n,  \label{1.eq1} \\
 \label{1.eq2}
  c_i\na\mu_i - \frac{c_i}{\sum_{k=1}^n c_k}\sum_{j=1}^n c_j\na\mu_j
	&= -\sum_{j=1}^n K_{ij}(\bm{c})c_ju_j, \\
  \sum_{j=1}^n c_ju_j &= 0, \label{1.eq3}
\end{align}
supplemented by the initial and boundary conditions
\begin{equation}\label{1.bic}
  \bm{c}(\cdot,0)=\bm{c}^0\quad\mbox{in }\Omega, \quad
	c_iu_i\cdot\nu = \na c_i\cdot\nu = 0\quad\mbox{on }\pa\Omega,\ t>0,\ i=1,\ldots,n,
\end{equation}
where $\Omega\subset\R^d$ ($d=1,2,3$) is a bounded domain, $\nu$ is the
exterior unit normal vector on the boundary $\pa\Omega$, $\bm{c}=(c_1,\ldots,c_n)$
is the density vector, and $K_{ij}(\bm{c})$ are the friction coefficients.
The left-hand side of \eqref{1.eq2} can be interpreted as the driving forces
of the thermodynamic system, and the right-hand side is the sum of the friction forces.
The chemical potentials 
\begin{equation}\label{1.mu}
  \mu_i=\frac{\delta\E}{\delta c_i} = \log c_i-\Delta c_i, \quad i=1,\ldots,n,
\end{equation}
are the variational derivatives of the (free) energy
\begin{equation}\label{1.HE}
  \E(\bm{c}) = \H(\bm{c}) + \frac{1}{2}\sum_{i=1}^n\int_\Omega|\na c_i|^2 dx, \quad
  \H(\bm{c}) = \sum_{i=1}^n\int_\Omega\big(c_i(\log c_i-1)+1\big)dx,
\end{equation}
and $\H(\bm{c})$ is the thermodynamic entropy. We assume that 
$\sum_{i=1}^n K_{ij}(\bm{c})=0$ for $j=1,\ldots,n$, meaning that the
linear system in $\na\mu_j$ is invertible only on a subspace, and that
$\sum_{i=1}^n c_i^0=1$ in $\Omega$, which implies that $\sum_{i=1}^n c_i(t)=1$
in $\Omega$ for all time $t>0$. This means that the mixture is saturated
and $c_i$ can be interpreted as volume fraction.
For simplicity, we have normalized all physical constants. 

Model \eqref{1.eq1}--\eqref{1.mu} has been derived rigorously in \cite{HJT19} 
in the high-friction
limit from a multicomponent Euler--Korteweg system for a general convex energy
functional depending on $\bm{c}$ and $\na\bm{c}$. A thermodynamics-based derivation
can be found in \cite{MiSc09}. When the energy equals
$\E(\bm{c})=\H(\bm{c})$, the model reduces to the classical Maxwell--Stefan
equations, analyzed first in \cite{Bot11,GoMa98,HMPW17} for local-in-time smooth
solutions and later in \cite{JuSt13} for global-in-time weak solutions.
In the single-species case, model \eqref{1.eq1}--\eqref{1.mu} becomes the
fourth-order Cahn--Hilliard equation with potential $\phi(c)=c(\log c-1)$,
which was analyzed in, e.g., \cite{ElGa96,Jin92}. Only few works are concerned
with the multi-species situation, and all of them require additional conditions. 
The mobility matrix in \cite{BoLa06,MaZi17} is assumed to be diagonal and that one in
\cite{KRS21} has constant entries, while the
works \cite{EMP21,ElGa97} suppose a particular (but nondiagonal)
structure of the mobility matrix. We also mention the works
\cite{BaEh18,BBEP20} on related models with free energies of the type $\H$.

The proof of the uniqueness of solutions to cross-diffusion or fourth-order 
systems is quite delicate due to the lack of a maximum principle and regularity 
of the solutions. 
The uniqueness of strong solutions to Maxwell--Stefan systems has been shown in
\cite{HMPW17,HuSa18}, and uniqueness results for weak solutions in a very 
special case can be found in \cite{ChJu18}. A weak-strong uniqueness result 
for Maxwell--Stefan systems was proved in \cite{HJT21}. 
Concerning uniqueness results for fourth-order equations, 
we refer to \cite{CGPS13} for single-species Cahn--Hilliard equations, 
\cite{Joh15} for single-species thin-film equations, 
and \cite{F13} for the quantum drift-diffusion equations. 
Up to our knowledge, there are no uniqueness results for multicomponent
Cahn--Hilliard systems. 
In this paper, we analyze these equations in a general setting for the first time.

\subsection{Key ideas of the analysis}

Before stating the main results, we explain the mathematical ideas needed
to analyze model \eqref{1.eq1}--\eqref{1.mu}. First, we rewrite
\eqref{1.eq2} 
by introducing the matrix $D(\bm{c})\in\R^{n\times n}$ with
entries 
$$
  D_{ij}(\bm{c}) = \frac{1}{\sqrt{c_i}} K_{ij}(\bm{c})\sqrt{c_j}
$$
in the unknowns $(\sqrt{c_1}u_1,\ldots,\sqrt{c_n}u_n)$:
\begin{equation}\label{1.D}
\begin{aligned}
  \sqrt{c_i}\na\mu_i - \frac{\sqrt{c_i}}{\sum_{k=1}^n c_k}\sum_{j=1}^n c_j\na\mu_j
	&= -\sum_{j=1}^n D_{ij}(\bm{c})\sqrt{c_j}u_j, \\
	\sum_{i=1}^n\sqrt{c_i}\big(\sqrt{c_i}u_i\big) &= 0.
\end{aligned}
\end{equation}
We show in Lemma \ref{lem.Dz} that this linear system has a unique solution
in the space $L(\bm{c}):=\{\bm{z}\in\R^n:\sum_{i=1}^n\sqrt{c_i}z_i=0\}$,
and the solution reads as
$$
  \sqrt{c_i}u_i = -\sum_{j=1}^n D_{ij}^{BD}(\bm{c})\sqrt{c_j}\na\mu_j,
$$
where $D^{BD}(\bm{c})$ is the so-called Bott--Duffin matrix inverse; see
Lemmas \ref{lem.Dz} and \ref{lem.DB} for the definition and some properties. Then,
defining the matrix $B(\bm{c})\in\R^{n\times n}$ with elements
\begin{equation}\label{1.B}
  B_{ij}(\bm{c}) = \sqrt{c_i}D_{ij}^{BD}(\bm{c})\sqrt{c_j}, \quad i,j=1,\ldots,n,
\end{equation}
system \eqref{1.eq1}--\eqref{1.eq2} can be formulated as 
(see Section \ref{sec.BD} for details)
$$
  \pa_t c_i = \diver\sum_{j=1}^n B_{ij}(\bm{c})\na\mu_j, \quad i=1,\ldots,n.
$$
The matrix $B(\bm{c})$ is often called Onsager or mobility matrix in the literature.
The major difficulty of the analysis consists in the fact 
that the matrix $B(\bm{c})$ is singular
and degenerates when $c_i\to 0$ for some $i\in\{1,\ldots,n\}$. 
Computing formally the energy identity
$$
  \frac{d\E}{dt}(\bm{c}) 
	+ \sum_{i,j=1}^n\int_\Omega B_{ij}(\bm{c})\na\mu_i\cdot\na\mu_j dx = 0,
$$
the degeneracy at $c_i=0$ prevents uniform estimates for $\na\mu_i$ in $L^2(\Omega)$.
In some works, this issue has been compensated. For instance, there exists
an entropy equality for the model of \cite{ElGa97} yielding an $L^2(\Omega)$
bound for $\Delta c_i$, and the decoupled mobilities in \cite{CMN19,MaZi17} allow for
decoupled entropy estimates. In our model, the energy identity does not
provide a gradient estimate for the full vector 
$(\na\mu_1,\ldots,\na\mu_n)$ but only for a projection:
$$
  \frac{d\E}{dt}(\bm{c}) + C_1\sum_{i=1}^n\int_\Omega
	\bigg|\sum_{j=1}^n(\delta_{ij}-\sqrt{c_ic_j})\sqrt{c_j}\na\mu_j\bigg|^2 dx \le 0,
$$
where $\delta_{ij}$
is the Kronecker delta; see Lemma \ref{lem.fei}.
(The constant $C_1>0$ and all constants that follow do not depend on $\bm{c}$.)
To address the degeneracy issue, we compute the time derivative of the entropy:
$$
  \frac{d\H}{dt}(\bm{c}) + \sum_{i,j=1}^n\int_\Omega B_{ij}(\bm{c})\na\log c_i
	\cdot\na\mu_j dx = 0.
$$
This does not provide a uniform estimate for $\Delta c_i$, but we show 
(see Lemma \ref{lem.fei}) that
$$
  \frac{d\H}{dt}(\bm{c}) + C_2\sum_{i=1}^n\int_\Omega(\Delta c_i)^2 dx
	\le C_3\sum_{i=1}^n\int_\Omega\bigg|\sum_{j=1}^n(\delta_{ij}-\sqrt{c_ic_j})
	\sqrt{c_j}\na\mu_j\bigg|^2 dx.
$$
Combining the energy and entropy inequalities in a suitable way, the last
integral cancels:
\begin{equation}\label{1.cee}
  \frac{d}{dt}\bigg(\H(\bm{c})+\frac{C_3}{C_1}\E(\bm{c})\bigg)
	+ C_2\sum_{i=1}^n\int_\Omega(\Delta c_i)^2 dx \le 0.
\end{equation}
This provides the desired $H^2(\Omega)$ bound for $c_i$.
Note that the energy or entropy inequality alone does not give estimates
for $c_i$. The combined energy-entropy inequality is the key idea of the paper
for both the existence and weak-strong uniqueness analysis.

\subsection{Main results}

We make the following assumptions:

\begin{itemize}
\item[(A1)] Domain: $\Omega\subset\R^d$ with $d\le 3$ is a bounded domain.
We set $Q_T=\Omega\times(0,T)$ for $T>0$.
\item[(A2)] Initial data: $c_i^0\in H^1(\Omega)$ satisfies $c_i^0\ge 0$ in $\Omega$,
$i=1,\ldots,n$, and $\sum_{i=1}^n c_i^0=1$ in $\Omega$. 
\end{itemize}

The assumption $d\le 3$ is made for convenience, it can be relaxed for higher space
dimension, by choosing another regularization in the existence proof;
see \eqref{3.regul1}. 
The constraint $\sum_{i=1}^n c_i^0=1$ expresses the saturation of the mixture
and it propagates to the solution.
%The solution vector $\bm c$ is assumed to satisfy the constraint  
%$\sum_{k=1}^n c_k = 1$; as the assumption on the data $c^0_i$  propagates to solutions.
We introduce the matrix $D_{ij}(\bm{c})=(1/\sqrt{c_i})K_{ij}(\bm{c})\sqrt{c_j}$
for $i,j=1,\ldots,n$ and set
\begin{equation}\label{1.L}
  L(\bm{c}) = \{\bm{x}\in\R^n:\sqrt{\bm{c}}\cdot\bm{x}=0\}, \quad
	L^\perp(\bm{c}) = \operatorname{span}\{\sqrt{\bm{c}}\},
\end{equation}
where $\sqrt{\bm{c}}=(\sqrt{c_1},\ldots,\sqrt{c_n})$.
The projections $P_L(\bm{c})$, $P_{L^\perp}(\bm{c})\in\R^{n\times n}$ 
on $L(\bm{c})$, $L(\bm{c})^\perp$, re\-spec\-tive\-ly, are given by
\begin{equation}\label{1.PL} 
  P_L(\bm{c})_{ij} = \delta_{ij}-\sqrt{c_ic_j}, \quad
	P_{L^\perp}(\bm{c})_{ij} = \sqrt{c_ic_j}\quad\mbox{for }i,j=1,\ldots,n.
\end{equation}
We impose for any given $\bm{c}\in[0,1]^n$ the following assumptions on 
$D(\bm{c})=(D_{ij}(\bm{c}))\in\R^{n\times n}$:

\begin{itemize}
\item[(B1)]$D(\bm{c})$ is symmetric and $\ran D(\bm{c})=L(\bm{c})$,
$\ker(D(\bm{c})P_L(\bm{c}))=L^\perp(\bm{c})$.
\item[(B2)] For all $i,j=1,\ldots,n$, $D_{ij}\in C^1([0,1]^n)$ is bounded.
\item[(B3)] The matrix $D(\bm{c})$ is positive semidefinite, and there exists 
$\rho>0$ such that all eigenvalues $\lambda\neq 0$ of
$D(\bm{c})$ satisfy $\lambda\ge\rho$.
\item[(B4)] For all $i,j=1,\ldots,n$, $K_{ij}(\bm{c})
=\sqrt{c_i}D_{ij}(\bm{c})/\sqrt{c_j}$ is bounded in $[0,1]^n$.
\end{itemize}

Examples of matrices $D(\bm{c})$ satisfying these assumptions
are presented in Section \ref{sec.exam}.
Our first main result is the global existence of weak solutions.

\begin{theorem}[Global existence]\label{thm.ex}
Let Assumptions (A1)--(A2) and (B1)--(B4) hold. Then there exists a weak solution
$\bm{c}$ to \eqref{1.eq1}--\eqref{1.mu} satisfying $0\le c_i\le 1$,
$\sum_{i=1}^n c_i=1$ in $\Omega\times(0,\infty)$, 
$$
  c_i\in L_{\rm loc}^\infty(0,\infty;H^1(\Omega))\cap 
	L_{\rm loc}^2(0,\infty;H^2(\Omega)), \quad
	\pa_t c_i\in L_{\rm loc}^2(0,\infty;H^1(\Omega)'),
$$
the initial condition in \eqref{1.bic} is satisfied in the sense of
$H^1(\Omega)'$, and for all $\phi_i\in C_0^\infty(\Omega\times(0,\infty))$,
\begin{align}\label{1.weak}
  0 &= -\int_0^\infty\int_\Omega c_i\pa_t\phi_i dxdt
	+ \sum_{j=1}^n\int_0^\infty\int_\Omega B_{ij}(\bm{c})\na\log c_i\cdot\na\phi_i dxdt \\
	&\phantom{xx}{}+ \sum_{j=1}^n\int_0^\infty\int_\Omega\diver(B_{ij}(\bm{c})\na\phi_i)
	\Delta c_j dxdt, \nonumber
\end{align}
where $B_{ij}(\bm{c})$ is defined in \eqref{1.B}. Furthermore,
\begin{align}\label{1.EH}
  \H(\bm{c}(\cdot,T)) &+ C_1\E(\bm{c}(\cdot,T)) 
	+ C_2\int_0^T\int_\Omega(|\na\sqrt{\bm{c}}|^2+|\Delta\bm{c}|^2) dxdt \\
	&{}+ C_2\int_0^T\int_\Omega|\bm{\zeta}|^2 dxdt
	%+ C_2\int_0^T\int_\Omega\bigg|\sum_{j=1}^n P_L(\bm{c})_{ij}\sqrt{c_j}
	%\na\mu_j\bigg|^2 dxdt
	\le \H(\bm{c}^0) + C_1\E(\bm{c}^0), \nonumber
\end{align}
where $C_1>0$ depends on $\rho$, $n$, $\|D(\bm{c})\|_F$ and $C_2>0$ depends
on $n$, $\|D(\bm{c})\|_F$ ($\|\cdot\|_F$ is the Frobenius matrix norm
and $\rho$ is introduced in Assumption (B3)). Moreover,
$\bm{\zeta}$ is the weak $L^2(\Omega)$ limit of an approximating sequence of
$\sum_{j=1}^n P_L(\bm{c})_{ij}\sqrt{c_j}\na\mu_j$.
\end{theorem}

Some comments are in order. First, by Assumption (B2), the elements of the matrix
$D(\bm{c})$ are bounded for any $\bm{c}\in[0,1]^n$ and therefore, the quantity
$\|D(\bm{c})\|_F$ is bounded uniformly in $\bm{c}$.
Second, the weak formulation \eqref{1.weak} makes sence 
since $B_{ij}(\bm{c})\na\log c_i \in L^2(Q_T)$. 
Indeed, by the definition of $B(\bm{c})$, we have
$$
  B_{ij}(\bm{c})\na\log c_j = \sqrt{c_i}D^{BD}_{ij}(\bm{c})\frac{1}{\sqrt{c_j}}\na c_j,
$$
and the matrix $\sqrt{c_i}D^{BD}_{ij}(\bm{c})/\sqrt{c_j}$ is bounded for all
$\bm{c}\in[0,1]^n$; see Lemma \ref{lem.DB} (iii) below. However, note that
the expression $\sum_{j=1}^n B_{ij}(\bm{c})\na\mu_j$ is generally not an element
of $L^2(Q_T)$. In particular, we cannot expect that $\na\Delta c_i\in L^2(Q_T)$.
Third, we have not been able to identify the weak limit $\bm{\zeta}$ because of 
low regularity. However, if 
$\sum_{j=1}^n P_L(\bm{c})_{ij}\sqrt{c_j}\na\mu_j\in 
L^2_{\rm loc}(0,\infty;L^2(\Omega))$ holds for all $i=1,\ldots,n$, then
we can identify $\zeta_i=\sum_{j=1}^n P_L(\bm{c})_{ij}\sqrt{c_j}\na\mu_j$;
see Lemma \ref{lem.ident}.

To prove Theorem \ref{thm.ex}, we first introduce a truncation with parameter
$\delta\in(0,1)$ as in \cite{ElGa97} to avoid the degeneracy. 
Then we reduce the cross-diffusion system to $n-1$
equations by replacing $c_n$ by $1-\sum_{i=1}^{n-1}c_i$. The advantage is that
the diffusion matrix of the reduced system is positive definite
(with a lower bound depending on $\delta$). The existence of solutions
$c_i^\delta$ to the truncated,
reduced system is proved by an approximation as in \cite{Jue16}
and the Leray--Schauder fixed-point theorem; see Section \ref{sec.approx}. 
An approximate version of the 
free energy estimate \eqref{1.EH} 
(proved in Lemma \ref{lem.eei} in Section \ref{sec.unif})
provides suitable uniform bounds that allow
us to perform the limit $\delta\to 0$. The approximate densities $c_i^\delta$
may be negative but, by exploiting the entropy bound for $c_i^\delta$,
its limit $c_i$ turns out to be nonnegative.
The limit $\delta\to 0$ is then performed in Section \ref{sec.exproof}, using
the uniform estimates and compactness arguments.

Our second main result is concerned with the weak-strong uniqueness.
For this, we define the relative entropy and free energy
in the spirit of \cite{GLT17} by, respectively, 
\begin{align}
  \H(\bm{c}|\bar{\bm{c}}) &:= \H(\bm{c}) - \H(\bar{\bm{c}})
	- \frac{\pa\H}{\pa\bm{c}}(\bar{\bm{c}})\cdot(\bm{c}-\bar{\bm{c}})
	= \sum_{i=1}^n\int_\Omega\bigg(c_i\log\frac{c_i}{\bar{c}_i} - (c_i-\bar{c}_i)
  \bigg)dx, \label{1.relH} \\
  \E(\bm{c}|\bar{\bm{c}}) &:= \E(\bm{c}) - \E(\bar{\bm{c}})
	- \frac{\pa\E}{\pa\bm{c}}(\bar{\bm{c}})\cdot(\bm{c}-\bar{\bm{c}})
	= \H(\bm{c}|\bar{\bm{c}}) + \frac12\sum_{i=1}^n\int_\Omega
	|\na(c_i-\bar{c}_i)|^2dx. \label{1.relE}
\end{align}

\begin{theorem}[Weak-strong uniqueness]\label{thm.wsu}
Let Assumptions (A1)--(A2), (B1)--(B4) hold, let $\bm{c}$ be a weak solution
to \eqref{1.eq1}--\eqref{1.mu} with initial datum $\bm{c^0}$, and let
$\bar{\bm{c}}$ be a strong solution to \eqref{1.eq1}--\eqref{1.mu} with
initial datum $\bar{\bm{c}}^0$. We assume that the weak solution
$\bm{c}$ satisfies
\begin{equation}\label{1.cregul}
  \sum_{j=1}^n P_L(\bm{c})_{ij}\sqrt{c_j}\na\mu_j
	\in L^2_{\rm loc}(0,\infty;L^2(\Omega))
	\mbox{ for }i,j=1,\ldots,n
\end{equation}
(see \eqref{1.PL} for the definition of $P_L(\bm{c})$)
and for all $T>0$ the energy and entropy inequalities
\begin{align}
  \E(\bm{c}(T)) + \sum_{i,j=1}^n\int_0^T\int_\Omega B_{ij}(\bm{c})
	\na\mu_i\cdot\na\mu_j dxdt &\le \E(\bm{c}^0), \label{1.dEdt} \\
	\H(\bm{c}(T)) + \sum_{i,j=1}^n\int_0^T\int_\Omega B_{ij}(\bm{c})
	\na\log c_i\cdot\na\mu_j dxdt &\le \H(\bm{c}^0). \label{1.dHdt}
\end{align}
The strong solution $\bar{\bm{c}}$ is supposed to be strictly positive,
i.e., there exists $m>0$ such that $\bar{c}_i\ge m$ in $\Omega$, $t>0$, and
satisfies the regularity
$$
  \bar{c}_i\in L_{\rm loc}^\infty(0,\infty;W^{3,\infty}(\Omega)),
	\quad \na\diver\bigg(\frac{1}{\bar{c}_i}B_{ij}(\bar{\bm{c}})\na\bar{\mu}_j\bigg)
	\in L^\infty_{\rm loc}(0,\infty;L^\infty(\Omega))
$$
for $i=1,\ldots,n$, 
as well as for any $T>0$ the energy and entropy conservation identities
\begin{align}
  \E(\bar{\bm{c}}(T)) + \sum_{i,j=1}^n\int_0^T\int_\Omega B_{ij}(\bar{\bm{c}})
	\na\bar{\mu}_i\cdot\na\bar{\mu}_j dxdt &= \E(\bar{\bm{c}}^0), \label{1.dEdtbar} \\
	\H(\bar{\bm{c}}(T)) + \sum_{i,j=1}^n\int_0^T\int_\Omega B_{ij}(\bar{\bm{c}})
	\na\log \bar{c}_i\cdot\na\bar{\mu}_j dxdt &= \H(\bar{\bm{c}}^0), 
\label{1.dHdtbar}
\end{align}
where $\mu_i=\log c_i-\Delta c_i$ and $\bar{\mu}_i=\log\bar{c}_i-\Delta \bar{c}_i$.
Then, for any $T>0$, there exist constants $C_1$, only depending on $\|D(\bm{c})\|_F$,
$n$, $\rho$, and $C_2(T)>0$, only depending on $T$, 
$\operatorname{meas}(\Omega)$, $n$, $\rho$, such that
\begin{equation}\label{1.comb}
  \H(\bm{c}(T)|\bar{\bm{c}}(T)) + C_1\E(\bm{c}(T)|\bar{\bm{c}}(T))
	\le C_2(T)\big(\H(\bm{c}^0|\bar{\bm{c}}^0) + C_1\E(\bm{c}^0|\bar{\bm{c}}^0)\big).
\end{equation}
In particular, if $\bm{c}^0=\bar{\bm{c}}^0$ then the weak and strong solutions
coincide.
\end{theorem}

Assumption \eqref{1.cregul} guarantees that the flux 
$\sum_{j=1}^n B_{ij}(\bm{c})\na\mu_j$ lies in $L^2(Q_T)$. Indeed, we prove in
Lemma \ref{lem.DB} (i) in Section \ref{sec.mobil} that $D_{ij}^{BD}(\bm{c})$
is bounded for $\bm{c}\in[0,1]^n$. Therefore, since 
$D^{BD}(\bm{c})=D^{BD}(\bm{c})P_L(\bm{c})$, assumption \eqref{1.cregul} 
and $c_i\in L^\infty(Q_T)$ imply that
\begin{equation}\label{1.regflux}
  \sum_{j=1}^n B_{ij}(\bm{c})\na\mu_j 
	= \sqrt{c_i}\sum_{j,k=1}^n D_{ik}^{BD}(\bm{c})P_L(\bm{c})_{kj}\sqrt{c_j}\na\mu_j
	\in L^2(Q_T).
\end{equation}
By the way, it follows from 
$\sum_{j=1}^n P_L(\bm{c})_{ij}\sqrt{c_j}\na\log c_j = 2\na\sqrt{c_i}\in L^2(Q_T)$ that
\begin{equation}\label{1.regc}
  \sum_{j=1}^n P_L(\bm{c})_{ij}\sqrt{c_j}\na\Delta c_j
	= \sum_{j=1}^n P_L(\bm{c})_{ij}\sqrt{c_j}\na(\log c_j-\mu_j) \in L^2(Q_T).
\end{equation}
Since $\na\Delta c_i$ may be not in $L^2(Q_T)$, we interpret \eqref{1.regc}
in the sense of distributions, i.e.\ for all $\Phi\in
C_0^\infty(\Omega;\R^d)$,
$$
  \bigg\langle\sum_{j=1}^n P_L(\bm{c})_{ij}\sqrt{c_j}\na\Delta c_j,\Phi
	\bigg\rangle
	= -\sum_{j=1}^n\int_\Omega\big(\na(P_L(\bm{c})_{ij}\sqrt{c_j})\cdot\Phi
	+ P_L(\bm{c})_{ij}\sqrt{c_j}\diver\Phi\big)\Delta c_j dx.
$$

For the proof of Theorem \ref{thm.wsu}, we estimate first the time derivative of
the relative entropy \eqref{1.relH}:
\begin{align*}
  \frac{d\H}{dt}&(\bm{c}|\bar{\bm{c}})
	+ C_1\sum_{i=1}^n\int_\Omega\bigg|\sum_{j=1}^n P_L(\bm{c})_{ij}
	\sqrt{c_j}\na\log\frac{c_j}{\bar{c}_j}\bigg|^2 dx
	+ C_1\sum_{i=1}^n\int_\Omega(\Delta(c_i-\bar{c}_i))^2 dx \\
	&\le C_2\sum_{i=1}^n\int_\Omega\bigg|\sum_{j=1}^n P_L(\bm{c})_{ij}
	\sqrt{c_j}\na(\mu_j-\bar{\mu}_j)\bigg|^2 dx + C_3\int_\Omega
	\E(\bm{c}|\bar{\bm{c}})dx,
\end{align*}
where $C_i>0$ are some constants depending only on the data.
The first term on the right-hand side can be handled by estimating the time
derivative of the relative energy \eqref{1.relE}:
\begin{align*}
  \frac{d\E}{dt}&(\bm{c}|\bar{\bm{c}}) + C_4\sum_{i=1}^n\int_\Omega
	\bigg|\sum_{j=1}^n P_L(\bm{c})_{ij}\sqrt{c_j}\na(\mu_j-\bar{\mu}_j)\bigg|^2 dx \\
	&\le \theta\sum_{i=1}^n\int_\Omega\bigg|\sum_{j=1}^n P_L(\bm{c})_{ij}
	\sqrt{c_j}\na\log\frac{c_j}{\bar{c}_j}\bigg|^2 dx
	+ \theta\sum_{i=1}^n\int_\Omega(\Delta(c_i-\bar{c}_i))^2 dx \\
	&\phantom{xx}{}+ C_5(\theta)\int_\Omega\E(\bm{c}|\bar{\bm{c}})dx,
\end{align*}
where $\theta>0$ can be arbitrarily small. Choosing $\theta=C_1C_4/C_2$, 
we can combine both estimates leading to 
$$
  \frac{d}{dt}\bigg(\H(\bm{c}|\bar{\bm{c}}) + \frac{C_2}{C_4}\E(\bm{c}|\bar{\bm{c}})
	\bigg)
	\le \bigg(C_3+\frac{C_2C_5}{C_4}\bigg)\E(\bm{c}|\bar{\bm{c}}),
$$
and the theorem follows after applying Gronwall's lemma. As the computations 
are quite involved, we compute first in Section \ref{sec.wsu.formal} 
the time derivative of the relative entropy and energy for {\em smooth} solutions.
The rigorous proof of the combined relative entropy-energy inequality 
for weak solutions $\bm{c}$ and strong solutions $\bar{\bm{c}}$ is
then performed in Section \ref{sec.wsu.rig}.

The paper is organized as follows. The Bott--Duffin matrix inverse is introduced
in Section \ref{sec.mobil}, some properties of the mobility matrix $B(\bm{c})$ are
proved, and the combined energy-entropy inequality \eqref{1.cee} is derived for
smooth solutions. The global existence of solutions (Theorem \ref{thm.ex}) is
shown in Section \ref{sec.ex}, while Section \ref{sec.wsu} is concerned with the
proof of the weak-strong uniqueness property (Theorem \ref{thm.wsu}). Finally,
we present some examples verifying Assumptions (B1)--(B4) in Section \ref{sec.exam}.

\subsection*{Notation}

Elements of the matrix $A\in\R^{n\times n}$ are denoted by $A_{ij}$,
$i,j=1,\ldots,n$, and the elements of a vector $\bm{c}\in\R^n$ are $c_1,\ldots,c_n$.
We use the notation $f(\bm{c})=(f(c_1),\ldots,f(c_n))$ for $\bm{c}\in\R^n$
and a function $f:\R\to\R$. The expression $|\na f(\bm{c})|^2$ is defined by
$\sum_{i=1}^n|\na f(c_i)|^2$ and $|\cdot|$ is the usual Euclidean norm.
The matrix $R(\bm{c})\in\R^{n\times n}$ 
is the diagonal matrix with elements $\sqrt{c_1},\ldots,\sqrt{c_n}$, i.e.\
$R_{ij}(\bm{c})=\sqrt{c_i}\delta_{ij}$ for $i,j=1,\ldots,n$, where
$\delta_{ij}$ denotes the Kronecker delta.
We understand by $\na\bm{\mu}$ the matrix with entries $\pa_{x_i}\mu_j$.
Furthermore, $C>0$, $C_i>0$ are generic constants 
with values changing from line to line.

%%%%%%%%%%%%%%%%%%%%%%%%%%%%%%%%%%%%%%%%%%%%%%%%%%%%%%%%%%%%%%%%%%%%%%%%%%%%%%

\section{Properties of the mobility matrix and a priori estimates}\label{sec.mobil}

We wish to express the fluxes $c_iu_i$ as a linear combination of the gradients
of the chemical potentials. Since $K(\bm{c})$ has a nontrivial kernel, 
we need to use a generalized matrix inverse, the Bott--Duffin inverse.
This inverse and its properties are studied in Section \ref{sec.BD}. The properties
allow us to derive in Section \ref{sec.apriori} some a priori estimates 
for the Maxwell--Stefan--Cahn--Hilliard system.

\subsection{The Bott--Duffin inverse}\label{sec.BD}

We wish to invert \eqref{1.eq2} or, equivalently, \eqref{1.D}. 
We recall definition \eqref{1.PL} of the projection matrices 
$P_L(\bm{c})\in\R^{n\times n}$ on $L(\bm{c})$ and 
$P_{L^\perp}(\bm{c})\in\R^{n\times n}$ on $L^\perp(\bm{c})$,
where $L(\bm{c})$ and $L^\perp(\bm{c})$ are defined in \eqref{1.L}.
Then \eqref{1.D} is equivalent to the problem:
\begin{equation}\label{2.Dz}
  \mbox{Solve}\quad D(\bm{c})\bm{z} = -P_L(\bm{c})R(\bm{c})\na\bm{\mu}
	\quad\mbox{in the space }\bm{z}\in L(\bm{c}),
\end{equation}
where $z_i=\sqrt{c_i}u_i$, recalling that 
$R(\bm{c})=\operatorname{diag}(\sqrt{\bm{c}})$.

\begin{lemma}[Solution of \eqref{2.Dz}]\label{lem.Dz}
Suppose that $D(\bm{c})$ satisfies Assumption (B1). The Bott--Duffin inverse
$$
  D^{BD}(\bm{c}) = P_L(\bm{c})\big(D(\bm{c})P_L(\bm{c})+P_{L^\perp}(\bm{c})\big)^{-1}
$$
is well-defined, symmetric, and satisfies $\ker D^{BD}(\bm{c})=L^\perp(\bm{c})$.
Furthermore, for any $\bm{y}\in L(\bm{c})$, the linear problem
$D(\bm{c})\bm{z}=\bm{y}$ for $\bm{z}\in L(\bm{c})$ has a unique solution
given by $\bm{z}=D^{BD}(\bm{c})\bm{y}$. 
\end{lemma}

We refer to \cite[Lemma 17]{HJT21} for the proof. The property for the kernel
follows from $\ker D^{BD}(\bm{c})=\ker P_L(\bm{c})=L^\perp(\bm{c})$.
Since $P_L(\bm{c})R(\bm{c})\na\bm{\mu}\in L(\bm{c})$
(this follows from the definition of $P_L(\bm{c})$ and $\sum_{i=1}^n c_i=1$),
we infer from Lemma \ref{lem.Dz} that \eqref{2.Dz} has the unique solution
$\bm{z}=-D^{BD}(\bm{c})P_L(\bm{c})R(\bm{c})\na\bm{\mu}\in L(\bm{c})$ or,
componentwise,
$$
  c_iu_i = \sqrt{c_i}z_i 
	= -\sum_{j=1}^n \sqrt{c_i}\big(D^{BD}(\bm{c})P_L(\bm{c})\big)_{ij}\sqrt{c_j}\na\mu_j 
	= -\sum_{j=1}^n \sqrt{c_i}D^{BD}(\bm{c})_{ij}\sqrt{c_j}\na\mu_j
$$
for $i=1,\ldots,n$, where the last equality follows from 
$D^{BD}(\bm{c})P_L(\bm{c})=D^{BD}(\bm{c})$; see \cite[(81)]{HJT21}.
Then we can formulate equation \eqref{1.eq1} as
\begin{equation}\label{2.B}
  \pa_t c_i = \diver\sum_{j=1}^n B_{ij}(\bm{c})\na\mu_j, \quad
	\mbox{where }B_{ij}(\bm{c})=\sqrt{c_i}D^{BD}_{ij}(\bm{c})\sqrt{c_j},
	\quad i,j=1,\ldots,n.
\end{equation}
The boundary conditions $c_iu_i\cdot\nu=0$ on $\pa\Omega$ yield
\begin{equation}\label{2.bc}
  \sum_{j=1}^n B_{ij}(\bm{c})\na\mu_j\cdot\nu = 0\quad\mbox{on }\pa\Omega,\ t>0,\
	i=1,\ldots,n.
\end{equation}

We recall some properties of the Bott--Duffin inverse.

\begin{lemma}[Properties of $D^{BD}(\bm{c})$]\label{lem.DB}
Suppose that $D(\bm{c})\in\R^{n\times n}$ satisfies Assumptions (B1)--(B4). Then:
\begin{itemize}
\item[\rm (i)] The coefficients $D^{BD}_{ij}\in C^1([0,1]^n)$ are bounded for
$i,j=1,\ldots,n$.
\item[\rm (ii)] Let $\lambda(\bm{c})$ be an eigenvalue of
$(D(\bm{c})P_L(\bm{c})+P_{L^\perp}(\bm{c}))^{-1}$. Then
$\evmin\le\lambda(\bm{c})\le\evmax$, where
$$
  \evmin = (1+n\|D(\bm{c})\|_F)^{-1}, \quad
	\evmax = \max\{1,\rho^{-1}\},
$$
$\|\cdot\|_F$ is the Frobenius matrix norm, and 
$\rho>0$ is a lower bound for the eigenvalues of $D(\bm{c})$; see Assumption (B3).
\item[\rm (iii)] The functions $\bm{c}\mapsto \sqrt{c_i}D^{BD}_{ij}(\bm{c})/\sqrt{c_j}$
are bounded in $[0,1]^n$ for $i,j=1,\ldots,n$.
\end{itemize}
\end{lemma}

A consequence of (ii) are the inequalities
\begin{equation}\label{2.DBD}
  \evmin|P_L(\bm{c})\bm{z}|^2 \le \bm{z}^T D^{BD}(\bm{c})\bm{z}
	\le \evmax|P_L(\bm{c})\bm{z}|^2 \quad\mbox{for }\bm{z}\in\R^n.
\end{equation}
Note that the Frobenius norm of $D(\bm{c})$ is bounded uniformly in $\bm{c}\in[0,1]^n$,
since $D_{ij}$ is bounded by Assumption (B1).

\begin{proof} The points (i) and (ii) are proved in \cite[Lemma 11]{HJT21}
in an interval $[m,1]^n$ for some $m>0$. In fact, 
we can conclude (i)--(ii) in the full interval $[0,1]^n$, 
since our Assumptions (B2)--(B3) are stronger than those in \cite{HJT21}.

For the proof of (iii),
dropping the argument $\bm{c}$ and observing that $RDR^{-1}=K$, we obtain
\begin{align*}
  R D^{BD}R^{-1} &= RP_L(DP_L+P_{L^\perp})^{-1}R^{-1}
	= RP_L(R^{-1}R)(DP_L+P_{L^\perp})^{-1}R^{-1} \\
	&= RP_LR^{-1}\big(R(DP_L+P_{L^\perp})R^{-1}\big)^{-1} \\
	&= RP_LR^{-1}\big(RDR^{-1}RP_LR^{-1}+RP_{L^\perp}R^{-1}\big)^{-1} \\
	&= RP_LR^{-1}\big(KRP_LR^{-1}+RP_{L^\perp}R^{-1}\big)^{-1}.
\end{align*}
The determinant of the expression in the brackets equals
$$
  \det\big(R(DP_L+P_{L^\perp})R^{-1}\big) = \det(DP_L+P_{L^\perp}).
$$
Therefore, denoting by ``adj'' the adjugate matrix, it follows that
\begin{equation}\label{2.RDR}
  R D^{BD}R^{-1} = \frac{RP_LR^{-1}\operatorname{adj}(KRP_LR^{-1}+RP_{L^\perp}R^{-1})}{
	\det(DP_L+P_{L^\perp})}.
\end{equation}
By Assumption (B3), the eigenvalues of $D$ are not smaller than $\rho>0$.
The proof of \cite[Lemma 11]{HJT21} shows that the eigenvalues of
$DP_L+P_{L^\perp}$ are not smaller than $\rho>0$, too. This implies that
$\det(DP_L+P_{L^\perp})\ge\rho^{n-1}>0$. The coefficients 
$$
  (RP_LR^{-1})_{ij} = \delta_{ij}-c_i, \quad (RP_{L^\perp}R^{-1})_{ij} = c_i
$$
are bounded for $\bm{c}\in[0,1]^n$ and, by Assumption (B4), the coefficients of
$K$ are also bounded. Therefore, all elements of 
$\operatorname{adj}(KRP_LR^{-1}+RP_{L^\perp}R^{-1})$ are bounded.
We conclude from \eqref{2.RDR} that the entries of $RD^{BD}R^{-1}$ are
bounded in $[0,1]^n$, i.e., point (iii) holds.
\end{proof}

The most important property is the positive definiteness of $D^{BD}(\bm{c})$
on $L(\bm{c})$; see \eqref{2.DBD}. This property implies the a priori estimates
proved in the following subsection.

\subsection{A priori estimates}\label{sec.apriori}

We show an energy inequality for smooth solutions.

\begin{lemma}[Free energy inequality]\label{lem.fei}
Let $\bm{c}\in C^\infty(\Omega\times(0,\infty);\R^n)$ be a positive, bounded, smooth
solution to \eqref{1.eq1}--\eqref{1.mu}. Then, for any 
$0<\lambda<\evmin$,
\begin{align*}
  \frac{d}{dt}\bigg(\H(\bm{c}) + \frac{(\evmax-\lambda)^2}{\evmin\lambda}
	\E(\bm{c})\bigg) &+ 2\lambda\int_\Omega|\na\sqrt{\bm{c}}|^2 dx
	+ \lambda\int_\Omega|\Delta\bm{c}|^2 dx \\
	&{}+ \frac{(\lambda_M-\lambda)^2}{2\lambda}\int_\Omega
	|P_L(\bm{c})R(\bm{c})\na\bm{\mu}|^2dx \le 0.
\end{align*}
where the entropy $\H(\bm{c})$ and the free energy $\E(\bm{c})$ are given by
\eqref{1.HE} and $\evmin$, $\evmax$ are defined in Lemma \ref{lem.DB}.
\end{lemma}

\begin{proof}
We derive first the energy inequality. To this end, we multiply equation
\eqref{2.B} for $c_i$ by $\mu_i=(\pa\E/\pa c_i)(\bm{c})$, integrate over $\Omega$,
integrate by parts (using the boundary conditions \eqref{2.bc}), and take
into account the lower bound \eqref{2.DBD} for $D^{BD}(\bm{c})$:
\begin{align}\label{2.dEdt}
  \frac{d\E}{dt}(\bm{c}) &= \sum_{i=1}^n\int_\Omega\frac{\pa\E}{\pa c_i}(\bm{c})
	\pa_t c_i dx 
	= -\sum_{i,j=1}^n\int_\Omega B_{ij}(\bm{c})\na\mu_i\cdot\na\mu_j dx \\
	&= -\sum_{i,j=1}^n D_{ij}^{BD}(\bm{c})(\sqrt{c_i}\na\mu_i)\cdot
	(\sqrt{c_j}\na\mu_j)dx
	\le -\evmin\int_\Omega|P_L(\bm{c})R(\bm{c})\na\bm{\mu}|^2 dx. \nonumber
\end{align}

The entropy inequality is derived by multiplying \eqref{2.B} by $\log c_i$, 
integrating over $\Omega$, and integrating by parts (using the boundary conditions 
\eqref{2.bc}):
\begin{equation*}
  \frac{d\H}{dt}(\bm{c}) 
	= \sum_{i=1}^n\int_\Omega(\log c_i)\pa_t c_i dx
	= -\sum_{i,j=1}^n\int_\Omega B_{ij}(\bm{c})\na\log c_i
	\cdot\na\mu_j dx.
\end{equation*}
To estimate the right-hand side, we
set $G=RP_LR$ (omitting the argument $\bm{c}$) and $M:=B-\lambda G$ for 
$\lambda\in(0,\evmin)$. Then
\begin{equation}\label{2.dHdt}
  \frac{d\H}{dt}(\bm{c}) = -\sum_{i,j=1}^n\int_\Omega M_{ij}\na\log c_i
	\cdot\na\mu_j dx - \lambda\sum_{i,j=1}^n\int_\Omega G_{ij}\na\log c_i
	\cdot\na\mu_j dx =: I_1 + I_2.
\end{equation}

Before estimating the integrals $I_1$ and $I_2$, we start with some preparations.
We use Lemma \ref{lem.DB} (ii) and $P_L^TP_L=P_L$ to obtain
$$
  \bm{z}^TB\bm{z} = (R\bm{z})^T D^{BD}R\bm{z} \ge \evmin|P_LR\bm{z}|^2
	= \evmin(P_LR\bm{z})^T (P_LR\bm{z}) = \evmin\bm{z}^T G\bm{z}\quad\mbox{for }
	\bm{z}\in\R^n.
$$
The matrix $M$ is positive semidefinite since for any $\bm{z}\in\R^n$,
\begin{equation}\label{psd.zMz}
  \bm{z}^T M\bm{z} = \bm{z}^TB\bm{z} - \lambda\bm{z}^T G\bm{z}
	\ge (\evmin-\lambda)\bm{z}^T G\bm{z} = (\evmin-\lambda)|P_LR\bm{z}|^2.
\end{equation}
Furthermore, by Lemma \ref{lem.DB} (ii) again, we have the upper bound
\begin{equation}\label{3.zMz}
  \bm{z}^T M\bm{z} = \bm{z}^T(B-\lambda G)\bm{z}
	\le (\evmax-\lambda)\bm{z}^TG\bm{z} = (\evmax-\lambda)|P_LR\bm{z}|^2.
\end{equation}

We are now in the position to estimate the integral $I_1$, 
using Young's inequality for any $\theta>0$:
\begin{align*}
  I_1 &\le \frac{\theta}{2}\sum_{i,j=1}^n\int_\Omega M_{ij}\na\log c_i
	\cdot\na\log c_j dx + \frac{1}{2\theta}\sum_{i,j=1}^n\int_\Omega
	M_{ij}\na\mu_i\cdot\na\mu_j dx \\
	&\le \frac{\theta}{2}(\evmax-\lambda)\int_\Omega|P_LR\na\log\bm{c}|^2 dx
	+ \frac{\evmax-\lambda}{2\theta}\int_\Omega|P_LR\na\bm{\mu}|^2 dx \\
	&= 2\theta(\evmax-\lambda)\int_\Omega|\na\sqrt{\bm{c}}|^2 dx
	+ \frac{\evmax-\lambda}{2\theta}\int_\Omega|P_LR\na\bm{\mu}|^2 dx,
\end{align*}
where the last step follows from $\sum_{j=1}^n (P_L)_{ij}R_j\na\log c_j
=2\na\sqrt{c_i}$, which is a consequence of $\sum_{j=1}^n\na c_j=0$.
For the integral $I_2$, we use the definitions $G_{ij} = c_i\delta_{ij}-c_ic_j$ and
$\mu_j=\log c_j-\Delta c_j$:
\begin{align*}
  I_2 &= -\lambda\sum_{i,j=1}^n\int_\Omega (c_i\delta_{ij}-c_ic_j)
	\frac{\na c_i}{c_i}\cdot\na(\log c_j-\Delta c_j) dx \\
	&= -\lambda\sum_{i=1}^n\int_\Omega\na c_i\cdot\na(\log c_i-\Delta c_i) dx
	+ \lambda\int_\Omega\sum_{i=1}^n \na c_i\cdot\sum_{j=1}^n 
	c_j\na(\log c_j-\Delta c_j)dx \\
	&= -\lambda\sum_{i=1}^n\int_\Omega\na c_i\cdot\na(\log c_i-\Delta c_i) dx
	= -\lambda\int_\Omega\big(4|\na\sqrt{\bm{c}}|^2 + |\Delta\bm{c}|^2\big)dx,
\end{align*}
where we integrated by parts in the last step.

Inserting the estimates for $I_1$ and $I_2$ into \eqref{2.dHdt} yields
\begin{align*}
  \frac{d\H}{dt}(\bm{c}) &+ 4\lambda\int_\Omega|\na\sqrt{\bm{c}}|^2 dx
	+ \lambda\int_\Omega|\Delta\bm{c}|^2 dx \\
	&\le 2\theta(\evmax-\lambda)\int_\Omega|\na\sqrt{\bm{c}}|^2 dx
	+ \frac{\evmax-\lambda}{2\theta}\int_\Omega|P_LR\na\bm{\mu}|^2 dx.
\end{align*}
We set $\theta=\lambda/(\evmax-\lambda)$ to conclude that
\begin{equation}\label{2.dHdt2}
  \frac{d\H}{dt}(\bm{c}) + 2\lambda\int_\Omega|\na\sqrt{\bm{c}}|^2 dx
	+ \lambda\int_\Omega|\Delta\bm{c}|^2 dx
	\le \frac{(\evmax-\lambda)^2}{2\lambda}\int_\Omega|P_LR\na\bm{\mu}|^2 dx.
\end{equation}
The right-hand side can be absorbed by the corresponding term in \eqref{2.dEdt}.
Indeed, adding the previous inequality to \eqref{2.dEdt} times
$(\evmax-\lambda)^2/(\evmin\lambda)$ finishes the proof.
\end{proof}

Note that the energy inequality \eqref{2.dEdt} or the entropy inequality
\eqref{2.dHdt2} alone are not sufficient to control the derivatives of $\bm{c}$ but
only a suitable linear combination. We will prove these inequalities rigorously
in the following section for weak solutions; see Lemma \ref{lem.eei}.

%%%%%%%%%%%%%%%%%%%%%%%%%%%%%%%%%%%%%%%%%%%%%%%%%%%%%%%%%%%%%%%%%%%%%%%%%%%%%%

\section{Proof of Theorem \ref{thm.ex}}\label{sec.ex}

We prove the existence of global weak solutions to \eqref{1.eq1}--\eqref{1.bic}.
For this, we construct an approximate system depending on a parameter $\delta>0$,
similarly as in \cite{ElGa97}, and then pass to the limit $\delta\to 0$.

\subsection{An approximate system}\label{sec.approx}

In order to deal with the degeneracy of the matrix $B(\bm{c})$ when a component
of $\bm{c}$ vanishes, we introduce the cutoff function $\chi_\delta:\R^n\to\R^n$ by
$$
  (\chi_\delta\bm{c})_i := \left\{\begin{array}{ll}
	\delta &\quad\mbox{for }c_i<\delta, \\
	c_i &\quad\mbox{for }\delta\le c_i\le 1-\delta, \\
	1-\delta &\quad\mbox{for }c_i>1-\delta,
	\end{array}\right.
$$
and define the approximate matrix
\begin{equation}\label{3.Bdelta}
  B^\delta(\bm{c}) := R(\chi_\delta\bm{c})D^{BD}(\chi_\delta\bm{c})R(\chi_\delta\bm{c}),
\end{equation}
recalling that $R(\chi_\delta\bm{c})=\operatorname{diag}(\sqrt{\chi_\delta\bm{c}})$.
We wish to solve the approximate problem
\begin{align}
  & \pa_t c_i^\delta = \diver\sum_{j=1}^n B_{ij}^\delta(\bm{c}^\delta)\na\mu_j^\delta,
	\quad \mu_j^\delta = \frac{\pa\E^\delta}{\pa c_j}(\bm{c}^\delta)
	\quad\mbox{in }\Omega,\ t>0, \label{3.eq} \\
  & c_i^\delta(\cdot,0)=c_i^0 \quad\mbox{in }\Omega, \quad
	\sum_{j=1}^n B_{ij}^\delta(\bm{c}^\delta)\na\mu_j^\delta\cdot\nu = 0,\
	\na c_i^\delta\cdot\nu=0	\quad\mbox{on }\pa\Omega, \label{3.bic}
\end{align}
where $i=1,\ldots,n$, $\sum_{i=1}^n c^0_i =1$ and the approximate energy is defined by
\begin{align}
  & \E^\delta(\bm{c}) := \H^\delta(\bm{c}) 
	+ \frac12\sum_{i=1}^n\int_\Omega|\na c_i|^2 dx,
	\quad \H^\delta(\bm{c}) := \sum_{i=1}^n\int_\Omega h_i^\delta(c_i)dx, \nonumber \\
	& h_i^\delta(r) = \left\{\begin{array}{ll}
	r\log\delta - \delta/2 + r^2/(2\delta) &\quad\mbox{for }r<\delta, \\
	r\log r &\quad\mbox{for }\delta\le r\le 1-\delta, \\
	r\log(1-\delta) - (1-\delta)/2 + r^2/(2(1-\delta)) 
	&\quad\mbox{for }r>1-\delta.
	\end{array}\right. \label{3.hdelta}
\end{align}
Observe that the solutions $c_i^\delta$ may be negative.
We will show below that $c_i^\delta$ converges to a nonnegative function as
$\delta\to 0$. The approximate entropy density is chosen in such a way that 
$h_i^\delta\in C^2(\R)$. Indeed, we obtain
$$
  (h_i^\delta)'(c_i) = \left\{\begin{array}{ll}
	\log\delta + c_i/\delta &\quad\mbox{for }c_i<\delta, \\
	\log c_i + 1 &\quad\mbox{for }\delta< c_i< 1-\delta, \\
	\log(1-\delta) + c_i/(1-\delta) &\quad\mbox{for }c_i>1-\delta,
	\end{array}\right. \quad
	(h_i^\delta)''(c_i) = \frac{1}{(\chi_\delta\bm{c})_i}.
$$
With these definitions, we obtain $\mu_i^\delta = (h_i^\delta)'(c_i^\delta)-\Delta
c_i^\delta$ for $i=1,\ldots,n$.

\begin{theorem}[Existence for the approximate system]\label{thm.approx}\quad
Let Assumptions (A1)--(A2) and (B1)--(B4) hold and let $\delta>0$. Then there
exists a weak solution $(\bm{c}^\delta,\bm{\mu}^\delta)$ to \eqref{3.eq}--\eqref{3.bic}
satisfying $\sum_{i=1}^n c_i^\delta(t)=1$ in $\Omega$, $t>0$,
\begin{align*}
  & c_i^\delta\in L_{\rm loc}^\infty(0,\infty;H^1(\Omega))\cap
	L_{\rm loc}^2(0,\infty;H^2(\Omega)), \\
	& \pa_t c_i\in L_{\rm loc}^2(0,\infty;H^2(\Omega)'), \quad
	\mu_i^\delta\in L_{\rm loc}^2(0,\infty;H^1(\Omega)), \quad i=1,\ldots,n,
\end{align*}
and the first equation in \eqref{3.eq} as well as the initial condition
in \eqref{3.bic} are satisfied in the sense of
$L_{\rm loc}^2(0,\infty;H^2(\Omega)')$.
%where $H^2_N(\Omega)=\{v\in H^2(\Omega):\na v\cdot\nu=0$ on $\pa\Omega\}$.
\end{theorem}

Before we prove this theorem, we show some properties of the matrix 
$B^\delta(\bm{c})$. We introduce the matrices
$P_L(\chi_\delta\bm{c})$, $P_{L^\perp}(\chi_\delta\bm{c})\in\R^{n\times n}$ 
with entries
$$
  P_L(\chi_\delta\bm{c})_{ij} = \delta_{ij} - \frac{\sqrt{(\chi_\delta\bm{c})_i
	(\chi_\delta\bm{c})_j}}{\sum_{k=1}^n(\chi_\delta\bm{c})_k}, \quad
	P_{L^\perp}(\chi_\delta\bm{c})_{ij} = \frac{\sqrt{(\chi_\delta\bm{c})_i
	(\chi_\delta\bm{c})_j}}{\sum_{k=1}^n(\chi_\delta\bm{c})_k}, \quad i,j=1,\ldots,n.
$$

\begin{lemma}[Properties of $B^\delta(\bm{c})$]\label{lem.Bdelta}\
Suppose that $D(\bm{c})$ satisfies Assumptions (B1)--(B4). Then Lemmas
\ref{lem.Dz} and \ref{lem.DB} hold with $P_L(\bm{c})$, $P_{L^\perp}(\bm{c})$, and
$D^{BD}(\bm{c})$ replaced by $P_L(\chi_\delta\bm{c})$, 
$P_{L^\perp}(\chi_\delta\bm{c})$, and $D^{BD}(\chi_\delta\bm{c})$. As a consequence,
the matrix $B^\delta(\bm{c})$, defined in \eqref{3.Bdelta}, satisfies
\begin{equation}\label{3.zBdeltaz}
  \bm{z}^TB^\delta(\bm{c})\bm{z}\ge\evmin|P_L(\chi_\delta\bm{c})R(\chi_\delta\bm{c})
	\bm{z}|^2 \quad\mbox{for any }\bm{z},\bm{c}\in\R^n,
\end{equation}
and the first $(n-1)\times(n-1)$ submatrix $\widetilde{B}^\delta(\bm{c})$ of
$B^\delta(\bm{c})$ is positive definite and satisfies for 
$\eta(\delta)=\evmin\delta^2/n$,
\begin{equation}\label{3.tilde}
  \widetilde{\bm{z}}^T\widetilde{B}^\delta(\bm{c})\widetilde{\bm{z}}
	\ge \eta(\delta)|\widetilde{\bm{z}}|^2\quad\mbox{for any }\widetilde{\bm{z}}
	\in\R^{n-1}.
\end{equation}
\end{lemma}

\begin{proof}
It can be verified that Assumptions (B1)--(B2) hold for $D(\chi_\delta\bm{c})$,
so Lemmas \ref{lem.Dz} and \ref{lem.DB} still hold for the matrix 
$D(\chi_\delta\bm{c})$. Inequality \eqref{3.zBdeltaz} is a direct consequence of
Lemma \ref{lem.DB} (ii). It remains to prove \eqref{3.tilde}.
We define for given $\widetilde{\bm{z}}\in\R^{n-1}$ the vector $\bm{z}\in\R^n$ with
$z_i=\widetilde{z}_i$ for $i=1,\ldots,n-1$ and $z_n=0$. Then \eqref{3.zBdeltaz}
becomes
\begin{equation}\label{3.aux}
  \widetilde{\bm{z}}^T\widetilde{B}^\delta(\bm{c})\widetilde{\bm{z}}
	\ge \evmin\big|\widetilde{P}_L(\chi_\delta\bm{c})\widetilde{R}(\chi_\delta\bm{c})
	\widetilde{\bm{z}}\big|^2 
	= \evmin\big(\widetilde{R}(\chi_\delta\bm{c})\widetilde{\bm{z}}\big)^T
	\widetilde{P}_L(\chi_\delta\bm{c})
	\big(\widetilde{R}(\chi_\delta\bm{c})\widetilde{\bm{z}}\big),
\end{equation}
where $\widetilde{A}$ denotes the first $(n-1)\times(n-1)$ submatrix of a given 
matrix $A\in\R^{n\times n}$. It follows from the Cauchy--Schwarz inequality that
for any $\zeta\in\R^{n-1}$,
\begin{align*}
  \zeta^T\widetilde{P}_L(\chi_\delta\bm{c})\zeta
	&= \sum_{i=1}^{n-1}\zeta_i^2 - \left(\sum_{j=1}^{n-1}
	\sqrt{\frac{(\chi_\delta\bm{c})_j}{\sum_{k=1}^n(\chi_\delta\bm{c})_k}}
	\zeta_j\right)^2
	\ge |\zeta|^2 - \sum_{j=1}^{n-1}\frac{(\chi_\delta\bm{c})_j}{\sum_{k=1}^n
	(\chi_\delta\bm{c})_k}|\zeta|^2 \\
	&= \frac{(\chi_\delta\bm{c})_n}{\sum_{k=1}^n(\chi_\delta\bm{c})_k}|\zeta|^2
	\ge \frac{\delta}{n}|\zeta|^2.
\end{align*}
Therefore, \eqref{3.aux} becomes
$$
  \widetilde{\bm{z}}^T\widetilde{B}^\delta(\bm{c})\widetilde{\bm{z}}
	\ge \frac{\evmin\delta}{n}\sum_{i=1}^{n-1}\big|\sqrt{(\chi_\delta\bm{c})_i}
	\widetilde{z}_i\big|^2 
	= \frac{\evmin\delta}{n}
	\sum_{i=1}^{n-1}(\chi_\delta\bm{c})_i\big|\widetilde{z}_i\big|^2 
	\ge \frac{\evmin\delta^2}{n}|\widetilde{\bm{z}}|^2,
$$
which proves \eqref{3.tilde}.
\end{proof}

We proceed to the proof of Theorem \ref{thm.approx}. The proof is divided
into four steps. First, we reformulate \eqref{3.eq} using the first $n-1$
components. Second, a time-discretized regularized system, similarly 
as in \cite[Chapter 4]{Jue16}, is constructed and the existence of weak solutions
to this system is proved. Third, we derive some uniform estimates from the
energy inequality. Finally, we perform the de-regularization limit.

{\em Step 1: Reformulation in $n-1$ components.}
We reformulate the approximate system in terms of the $n-1$ relative chemical
potentials
$$
  w_i^\delta = \mu_i^\delta-\mu_n^\delta, \quad i=1,\ldots,n-1.
$$
It holds that 
$$
  \sum_{j=1}^n \big(P_L(\chi_\delta\bm{c})R(\chi_\delta\bm{c})\big)_{kj}
	= \sum_{j=1}^n\bigg(\delta_{kj} - \frac{\sqrt{(\chi_\delta\bm{c})_k
	(\chi_\delta\bm{c})_j}}{\sum_{\ell=1}^n(\chi_\delta\bm{c})_\ell}\bigg)
	\sqrt{(\chi_\delta\bm{c})_j} = 0.
$$
Then, using $D^{BD}(\bm{c})=D^{BD}(\bm{c})P_L(\bm{c})$ (which is a general
property of the Bott--Duffin inverse; see \cite[(81)]{HJT21}),
\begin{align*}
  \sum_{j=1}^n B_{ij}^\delta(\bm{c}) &= \sum_{j=1}^n\sqrt{(\chi_\delta\bm{c})_i}
	D_{ij}^{BD}(\bm{c})\sqrt{(\chi_\delta\bm{c})_j} \\
	&= \sum_{j,k=1}^n\sqrt{(\chi_\delta\bm{c})_i}D_{ik}^{BD}(\bm{c})
	\big(P_L(\chi_\delta\bm{c})R(\chi_\delta\bm{c})\big)_{kj} = 0.
\end{align*}
This shows that
$$
  \sum_{j=1}^n B_{ij}^\delta(\bm{c})\na\mu_j^\delta
	= \sum_{j=1}^{n-1}B_{ij}^\delta(\bm{c})\na\mu_j^\delta 
	+ B_{in}^\delta(\bm{c})\na\mu_n^\delta 
	%&= \sum_{j=1}^{n-1}B_{ij}^\delta(\bm{c})\na\mu_j^\delta
	%- \sum_{j=1}^{n-1}B_{ij}^\delta(\bm{c})\na\mu_n^\delta
	= \sum_{j=1}^{n-1}B_{ij}^\delta(\bm{c})\na(\mu_j^\delta-\mu_n^\delta).
$$
Consequently, we can rewrite the first equation in \eqref{3.eq} as
\begin{equation}\label{3.cdelta}
  \pa_t c_i^\delta = \diver\sum_{j=1}^{n-1}\widetilde{B}_{ij}^\delta(\bm{c}^\delta)
	\na w_j^\delta, \quad i=1,\ldots,n-1, \quad 
	c_n^\delta = 1-\sum_{i=1}^{n-1}c_i^\delta,
\end{equation}
recalling that $\widetilde{B}^\delta$ is the first $(n-1)\times(n-1)$ submatrix
of $B^\delta$.

{\em Step 2: Existence for a regularized system.}
We consider for given $\delta>0$, $T>0$, $N\in\N$, and 
$(c_1^{k-1},\ldots,c_{n-1}^{k-1})$ the regularized system
\begin{align}\label{3.regul1}
  & \frac{1}{\tau}(c_i^k-c_i^{k-1}) 
	= \diver\sum_{j=1}^{n-1}\widetilde{B}_{ij}^\delta(  \widetilde{\bm{c}}^k )
	\na w_j^k - \eps(\Delta^2 w_i^k + w_i^k)\quad\mbox{in }\Omega, \\
	& w_i^k = (h_i^\delta)'(c_i^k) - (h_n^\delta)'(c_n^k)
	- \Delta(c_i^k-c_n^k), \quad i=1,\ldots,n-1, \label{3.regul2}
\end{align}
where $\tau=T/N$ and $c_n^k=1-\sum_{i=1}^{n-1}c_i^k$. 
Equation \eqref{3.regul1} is understood in the weak sense
$$
  \frac{1}{\tau} \int_\Omega(c_i^k-c_i^{k-1})\phi_i dx 
	+ \sum_{j=1}^{n-1} \int_\Omega \widetilde{B}_{ij}^\delta(\bm{c}^k)
	\na\phi_i\cdot\na w_j^k dx
	+ \eps \int_\Omega(\Delta w_i^k\Delta\phi_i + w_i^k\phi_i)dx = 0
$$
for test functions $\phi_i\in H^2(\Omega)$.

The $\eps$-regularization ensures that $w_i^k\in H^2(\Omega)\hookrightarrow
L^\infty(\Omega)$ since $d\le 3$. In higher space dimensions, we can replace
$\Delta^2 w_i^k$ by $(-\Delta)^m w_i^k$ with $m>d/2$, which gives
$w_i^k\in H^m(\Omega)\hookrightarrow L^\infty(\Omega)$.

We prove the solvability of \eqref{3.regul1}--\eqref{3.regul2} in two steps.

\begin{lemma}[Solvability of \eqref{3.regul2}]\label{lem.solv}
Let $\bm{w}\in L^2(\Omega;\R^{n-1})$.
Then there exists a unique strong solution $\widetilde{\bm{c}}\in H^2(\Omega;\R^{n-1})$ 
to
\begin{equation}\label{3.deltac}
  w_i = (h_i^\delta)'(c_i) - (h_n^\delta)'(c_n)	- \Delta(c_i-c_n)\quad\mbox{in }\Omega,
	\quad \na c_i\cdot\nu=0\quad\mbox{on }\pa\Omega
\end{equation}
for $i=1,\ldots,n-1$, where $c_n=1-\sum_{i=1}^{n-1}c_i$. This defines the
operator ${\mathcal L}:L^2(\Omega;\R^{n-1})
\to H^2(\Omega;\R^{n-1})$, ${\mathcal L}(\bm{w})=\widetilde{\bm{c}}$.
\end{lemma}

\begin{proof}
The system of equations can be written as
$$
  \diver(M\na\widetilde{\bm{c}})_i = (h_i^\delta)'(c_i) - (h_n^\delta)'(c_n) - w_i
	\quad\mbox{in }\Omega,
$$
where the entries of the diffusion matrix $M$ are $M_{ii}=2$ and
$M_{ij}=1$ for all $i\neq j$. In particular, $M$ is symmetric and positive definite. 
Thus, we can apply
the theory for elliptic systems with sublinear growth function and conclude the
existence of a unique weak solution $\widetilde{\bm{c}}\in H^1(\Omega;\R^{n-1})$.
It remains to verify that this solution lies in $H^2(\Omega;\R^{n-1})$.
Summing \eqref{3.deltac} over $i=1,\ldots,n-1$, we find that
$$
  \Delta c_n = -\sum_{i=1}^{n-1}\Delta c_i 
	= \frac{1}{n}\sum_{i=1}^{n-1}(w_i-(h_i^\delta)'(c_i)) 
	+ \frac{n-1}{n} (h_n^\delta)'(c_n)
	\in L^2(\Omega)
$$
with the boundary condition $\na c_n\cdot\nu=0$ on $\pa\Omega$.
We infer from elliptic regularity theory that $c_n\in H^2(\Omega)$.
Consequently, $\Delta c_n\in L^2(\Omega)$ and elliptic regularity again
implies that $c_i\in H^2(\Omega)$.
\end{proof}

It follows from Lemma \ref{lem.solv} that we can write \eqref{3.regul1} as
\begin{equation}\label{3.Bregul}
  \frac{1}{\tau}({\mathcal L}(\bm{w})_i - c_i^{k-1})
	= \diver\sum_{j=1}^{n-1} \widetilde{B}_{ij}^\delta(\widetilde{\bm{c}}^k)
	\na w_j^k - \eps(\Delta^2 w_i^k+w_i^k)\quad\mbox{in }\Omega,\ i=1,\ldots,n-1.
\end{equation}

\begin{lemma}[Solvability of \eqref{3.Bregul}]
Let $\widetilde{\bm{c}}^{k-1}\in H^2(\Omega;\R^{n-1})$.
Then there exists a weak solution $\bm{w}^k\in H^2(\Omega;\R^{n-1})$ 
to \eqref{3.Bregul} such that for all $\phi_i\in L^2(0,T;H^2(\Omega))$,
\begin{align*}
  \frac{1}{\tau}\int_\Omega({\mathcal L}(\bm{w})_i - c_i^{k-1})\phi_i dx
	&+ \sum_{i,j=1}^{n-1}\int_\Omega\widetilde{B}_{ij}^\delta({\mathcal L}(\bm{w}))
	\na\phi_i \cdot\na w_j^k dx \\
	&{}+ \eps\sum_{i=1}^{n-1}\int_\Omega(\Delta w_i^k\Delta\phi_i+w_i^k\phi_i)dx = 0.
\end{align*} 
\end{lemma}

\begin{proof}
Given $\bar{\bm{w}}\in L^\infty(\Omega;\R^{n-1})$ and $\sigma\in[0,1]$, 
we wish to find a solution to the linear problem
\begin{equation}\label{3.LM}
  {\mathcal A}(\bm{w},\bm{\phi}) = {\mathcal F}(\bm{\phi})\quad\mbox{for }
	\bm{\phi}\in H^2(\Omega;\R^{n-1}),
\end{equation}
where
\begin{align*}
  {\mathcal A}(\bm{w},\bm{\phi}) &= \sum_{i,j=1}^{n-1}\int_\Omega
	\widetilde{B}_{ij}^\delta({\mathcal L}(\bar{\bm{w}}))\na\phi_i\cdot\na w_j dx
	+ \eps\sum_{i=1}^{n-1}\int_\Omega(\Delta w_i\Delta\phi_i + w_i\phi_i)dx, \\
	{\mathcal F}(\bm{\phi}) &= -\frac{\sigma}{\tau}\int_\Omega({\mathcal L}(\bar{\bm{w}})
	- \widetilde{\bm{c}}^{k-1})\cdot\bm{\phi}dx.
\end{align*}
We infer from the boundedness of $\widetilde{B}_{ij}^\delta
({\mathcal L}(\bar{\bm{w}}))$
that the bilinear form ${\mathcal A}$ is continuous on $H^2(\Omega;\R^{n-1})$.
Furthermore, by the positive definiteness of 
$\widetilde{B}_{ij}^\delta({\mathcal L}(\bar{\bm{w}}))$, thanks to \eqref{3.tilde},
${\mathcal A}$ is coercive. Moreover, ${\mathcal F}$ is a continuous linear form
on $H^2(\Omega;\R^{n-1})$. We conclude from the Lax--Milgram theorem that there
exists a unique solution $\bm{w}\in H^2(\Omega;\R^{n-1})$ to \eqref{3.LM}.
Since $d\le 3$ by Assumption (A1), we have $H^2(\Omega)\hookrightarrow
L^\infty(\Omega)$ and therefore $\bm{w}\in L^\infty(\Omega;\R^{n-1})$.

This defines the fixed-point operator $S:L^\infty(\Omega;\R^{n-1})\times[0,1]\to
L^\infty(\Omega;\R^{n-1})$, $S(\bar{\bm{w}},\sigma)=\bm{w}$. The operator $S$ 
is continuous, and it satisfies $S(\bar{\bm{w}},0)=\bm{0}$ for all $\bar{\bm{w}}
\in L^\infty(\Omega;\R^{n-1})$. In view of the compact embedding
$H^2(\Omega)\hookrightarrow L^\infty(\Omega)$, $S$ is also compact. It remains to
verify that all fixed points of $S(\cdot,\sigma)$ are uniformly bounded.
To this end, let $\bm{w}\in L^\infty(\Omega;\R^{n-1})$ be such a fixed point.
Then $\bm{w}\in H^2(\Omega;\R^{n-1})$ solves \eqref{3.LM} with 
$\bar{\bm{w}}=\bm{w}$. We choose the test function $\bm{\phi} = \bm{w}$ in \eqref{3.LM}
to find that
\begin{equation}\label{3.aux2}
  \frac{\sigma}{\tau}\int_\Omega(\widetilde{\bm{c}}-\widetilde{\bm{c}}^{k-1})
	\cdot\bm{w} dx + \sum_{i,j=1}^{n-1}\int_\Omega \widetilde{B}_{ij}^\delta
	(\widetilde{\bm{c}})\na w_i\cdot\na w_j dx
	+ \eps\sum_{i=1}^{n-1}\int_\Omega((\Delta w_i)^2+w_i^2)dx = 0,
\end{equation}
where $\widetilde{\bm{c}}={\mathcal L}(\bm{w})=(c_1,\ldots,c_{n-1})$ and $c_i$ solves 
\eqref{3.regul2} with $w_i^k$ replaced by $w_i$. Using the test function
$c_i-c_i^{k-1}$ in the weak formulation of \eqref{3.regul2} leads to
\begin{align*}
  \sum_{i=1}^{n-1}\int_\Omega(c_i-c_i^{k-1})w_i dx
	&= \sum_{i=1}^{n-1}\int_\Omega\big(\na(c_i-c_n)\cdot\na(c_i-c_i^{k-1}) \\ 
	&\phantom{xx}{}+ ((h_i^\delta)'(c_i)-(h_i^\delta)'(c_n))(c_i-c_i^{k-1})\big)dx.
\end{align*}
The convexity of the function $h_i^\delta$ and $\sum_{i=1}^{n-1}c_i=1-c_n$ imply that
\begin{align*}
  \sum_{i=1}^{n-1}(c_i-c_i^{k-1})(h_i^\delta)'(c_i)
	&\ge \sum_{i=1}^{n-1}\big(h_i^\delta(c_i)-h_i^\delta(c_i^{k-1})\big), \\
	-\sum_{i=1}^{n-1}(c_i-c_i^{k-1})(h_n^\delta)'(c_n)
	&= (c_n-c_n^{k-1})(h_n^\delta)'(c_n) \ge h_n^\delta(c_n)-h_n^\delta(c_n^{k-1}).
\end{align*}
Moreover, since $\sum_{i=1}^{n-1}\na c_i=-\na c_n$ and 
$\sum_{i=1}^{n-1}\na c_i^{k-1}=-\na c_n^{k-1}$,
\begin{align*}
  \sum_{i=1}^{n-1}\na(c_i-c_n)\cdot\na (c_i-c_i^{k-1})
	&= \sum_{i=1}^{n}|\na c_i|^2 - \sum_{i=1}^{n}\na c_i^{k-1}\cdot\na c_i \\ 
	&\ge \frac12\sum_{i=1}^n|\na c_i|^2 - \frac12\sum_{i=1}^n|\na c_i^{k-1}|^2.
\end{align*}
This yields
\begin{align*}
  \sum_{i=1}^{n-1}\int_\Omega(c_i-c_i^{k-1})w_i dx 
	&\ge \sum_{i=1}^{n}\int_\Omega\big(h_i^\delta(c_i)-h_i^\delta(c_i^{k-1})\big)dx 
	+ \frac12\sum_{i=1}^{n}\int_\Omega\big(|\na c_i|^2 - |\na c_i^{k-1}|^2\big) dx \\
	&\ge \widetilde{\E}^\delta(\widetilde{\bm{c}}) 
	- \widetilde{\E}^\delta(\widetilde{\bm{c}}^{k-1}),
\end{align*}
where
$$
  \widetilde{\E}^\delta(\widetilde{\bm{c}}) := \widetilde{\H}^\delta(\widetilde{\bm{c}})
	+ \sum_{i=1}^n\int_\Omega|\na c_i|^2 dx, \quad
	\widetilde{\H}^\delta(\widetilde{\bm{c}}) := \H^\delta(\bm{c}).
$$
Inserting this inequality into \eqref{3.aux2} finally gives
\begin{equation}\label{3.rfei}
  \sigma\widetilde{\E}^\delta(\widetilde{\bm{c}})
	+ \tau\sum_{i,j=1}^{n-1}\int_\Omega \widetilde{B}_{ij}^\delta(\widetilde{\bm{c}})
	\na w_i\cdot\na w_j dx + \eps\tau\int_\Omega(|\Delta\bm{w}|^2 + |\bm{w}|^2)dx
	\le \sigma\widetilde{\E}^\delta(\widetilde{\bm{c}}^{k-1}).
\end{equation}
By the positive definiteness of $\widetilde{B}^\delta$ (positive semidefiniteness
is sufficient), this gives a uniform $H^2(\Omega)$ bound and
consequently a uniform $L^\infty(\Omega)$ bound for $\bm{w}$.
The Leray--Schauder fixed-point theorem now implies the existence of a solution
to \eqref{3.regul1}--\eqref{3.regul2}.
\end{proof}

{\em Step 3: Uniform estimates.}
We wish to derive estimates uniform in $\eps$ and $\tau$. The starting point is
the regularized energy estimate \eqref{3.rfei} and the positive definiteness
estimate \eqref{3.tilde}. First, we introduce the piecewise constant in time functions
$\bm{w}^{(\tau)}(x,t)=\bm{w}^k(x)$, $\widetilde{\bm{c}}^{(\tau)}(x,t)
= {\mathcal L}(\bm{w}^k(x))$ for $x\in\Omega$ and $t\in((k-1)\tau,k\tau]$,
$k=1,\ldots,N$, and set 
$\bm{w}^{(\tau)}(x,0)=(\pa\widetilde{\E}/\pa\widetilde{\bm{c}})(\widetilde{\bm{c}}^0)$ 
and $\widetilde{\bm{c}}^{(\tau)}(x,0)=\widetilde{\bm{c}}^0$. Introducing the shift
operator $(\sigma_\tau\bm{w}^{(\tau)})(x,t)=\bm{w}^{(\tau)}(x,t-\tau)$ for
$x\in\Omega$ and $t\ge \tau$, we can formulate \eqref{3.regul1}--\eqref{3.regul2} as
\begin{align}
  & \frac{1}{\tau}(\widetilde{\bm{c}}^{(\tau)}-\sigma_\tau\widetilde{\bm{c}}^{(\tau)})
	= \diver(\widetilde{B}^\delta(\widetilde{\bm{c}})\na\bm{w}^{(\tau)})
	- \eps(\Delta^2\bm{w}^{(\tau)}+\bm{w}^{(\tau)}), \label{3.tau1} \\
	& w_i^{(\tau)} = (h_i^\delta)'(c_i^{(\tau)}) - (h_n^\delta)'(c_n^{(\tau)})
	- \Delta(c_i^{(\tau)}-c_n^{(\tau)}), \quad i=1,\ldots,n-1, \label{3.tau2} 
\end{align}
recalling that $\widetilde{\bm{c}}^{(\tau)}={\mathcal L}(\bm{w}^{(\tau)})$
is a function of $\bm{w}^{(\tau)}$. Then \eqref{3.rfei} can be 
written after summation over $k=1,\ldots,N$ as
\begin{equation*}
  \widetilde{\E}^\delta(\widetilde{\bm{c}}^{(\tau)}(T))
	+ \eta(\delta)\int_0^T\int_\Omega|\na\bm{w}^{(\tau)}|^2 dxdt
	+ \eps C\int_0^T\|\bm{w}^{(\tau)}\|_{H^2(\Omega)}^2 dt
	\le \widetilde{\E}^\delta(\widetilde{\bm{c}}^0),
\end{equation*}
where we used \eqref{3.tilde} and the generalized Poincar\'e inequality with
constant $C>0$. This implies the estimates
\begin{equation}\label{3.west}
  C(\delta)\|\bm{w}^{(\tau)}\|_{L^2(0,T;H^1(\Omega))}
	+ \sqrt{\eps}\|\bm{w}^{(\tau)}\|_{L^2(0,T;H^2(\Omega))} \le C,
\end{equation}
where $C>0$ denotes here and in the following a constant independent of
$\eps$ and $\tau$. 

To derive a uniform estimate for $\widetilde{\bm{c}}^{(\tau)}$, we multiply
\eqref{3.tau2} by $-\Delta c_i^{(\tau)}$, integrate over $Q_T=\Omega\times(0,T)$,
integrate by parts, and sum over $i=1,\ldots,n-1$:
\begin{align*}
  \sum_{i=1}^{n-1}\int_0^T\int_\Omega&\na w_i^{(\tau)}\cdot\na c_i^{(\tau)} dxdt
	= \sum_{i=1}^{n-1}\int_0^T\int_\Omega\na\big((h_i^\delta)'(c_i^{(\tau)})
	- (h_n^\delta)'(c_n^{(\tau)})\big)\cdot\na c_i^{(\tau)} dxdt \\
	&\phantom{xx}{}+ \sum_{i=1}^{n-1}\int_0^T\int_\Omega\big((\Delta c_i^{(\tau)})^2
	- \Delta c_i^{(\tau)}\Delta c_n^{(\tau)}\big)dxdt =: I_3+I_4.
\end{align*}
Since $\na(h_i^\delta)'(c_i^{(\tau)})=(h_i^\delta)''(c_i^{(\tau)})\na c_i^{(\tau)}
= \na c_i^{(\tau)}/(\chi_\delta\bm{c}^{(\tau)})_i$ and 
$\sum_{i=1}^{n-1}\na c_i^{(\tau)}=-\na c_n^{(\tau)}$, the term $I_3$ can be written as
$$
  I_3 = \sum_{i=1}^n\int_0^T\int_\Omega
	\frac{|\na c_i^{(\tau)}|^2}{(\chi_\delta\bm{c}^{(\tau)})_i}dxdt.
$$
Using the property $\sum_{i=1}^{n-1}\Delta c_i^{(\tau)}=-\Delta c_n^{(\tau)}$, 
the remaining term $I_4$ becomes
$$
  I_4 = \sum_{i=1}^n\int_0^T\int_\Omega(\Delta c_i^{(\tau)})^2 dxdt.
$$
Therefore, by Young's inequality,
\begin{align*}
  \sum_{i=1}^n&\int_0^T\int_\Omega(\Delta c_i^{(\tau)})^2 dxdt
	+ \sum_{i=1}^n\int_0^T\int_\Omega
	\frac{|\na c_i^{(\tau)}|^2}{(\chi_\delta\bm{c}^{(\tau)})_i}dxdt
	= \sum_{i=1}^{n-1}\int_0^T\int_\Omega\na w_i^{(\tau)}\cdot\na c_i^{(\tau)}dxdt \\
	&\le \frac12\sum_{i=1}^{n-1}\int_0^T\int_\Omega\bigg(\frac{|\na c_i^{(\tau)}|^2}{
	(\chi_\delta\bm{c}^{(\tau)})_i}
	+ (\chi_\delta\bm{c}^{(\tau)})_i|\na w_i^{(\tau)}|^2\bigg) dxdt \\
	&\le \frac12\sum_{i=1}^{n-1}\int_0^T\int_\Omega\frac{|\na c_i^{(\tau)}|^2}{
	(\chi_\delta\bm{c}^{(\tau)})_i}dxdt
	+ \frac12\sum_{i=1}^{n-1}\int_0^T\int_\Omega|\na w_i^{(\tau)}|^2dxdt.
\end{align*}
The first term on the right-hand side is absorbed by the left-hand side. Thus,
we deduce from \eqref{3.west} that
$$
  \sum_{i=1}^n\int_0^T\int_\Omega(\Delta c_i^{(\tau)})^2 dxdt
	+ \frac12\sum_{i=1}^n\int_0^T\int_\Omega
	\frac{|\na c_i^{(\tau)}|^2}{(\chi_\delta\bm{c}^{(\tau)})_i}dxdt
	\le \frac12\|\na \bm{w}^{(\tau)}\|^2_{L^2(Q_T)}
	\le C.
$$
Since $c_i^{(\tau)}\in L^\infty(Q_T)$, we infer from the
previous estimate that
\begin{equation}\label{3.cH2}
  \|c_i^{(\tau)}\|_{L^2(0,T;H^2(\Omega))} \le C, \quad i=1,\ldots,n.
\end{equation}

Finally, we derive an estimate for the discrete time derivative. It follows from
\eqref{3.Bregul} that
\begin{align*}
  \frac{1}{\tau}\|c_i^{(\tau)}-\sigma_\tau c_i^{(\tau)}\|_{L^2(0,T;H^2(\Omega)')}
	&\le \sum_{j=1}^{n-1}\|\widetilde{B}^\delta_{ij}(\widetilde{\bm{c}}^{(\tau)})
	\|_{L^\infty(Q_T)}\|\na w_j^{(\tau)}\|_{L^2(Q_T)} \\
	&\phantom{xx}{}+ \eps\|w_i^{(\tau)}\|_{L^2(0,T;H^2(\Omega))}.
\end{align*}
The entries of $\widetilde{B}^\delta(\widetilde{\bm{c}}^{(\tau)})$ are bounded since
$\delta\le(\chi_\delta\bm{c}^{(\tau)})_i\le 1-\delta$. Thus, by \eqref{3.west},
\begin{equation}\label{3.ctime}
  \tau^{-1}\|c_i^{(\tau)}-\sigma_\tau c_i^{(\tau)}\|_{L^2(0,T;H^2(\Omega)')}
	\le C, \quad i=1,\ldots,n-1.
\end{equation}

{\em Step 4: Limit $(\eps,\tau)\to 0$.}
In view of estimates \eqref{3.cH2} and \eqref{3.ctime}, we can apply the
Aubin--Lions lemma in the version of \cite[Theorem 1]{DrJu12} to conclude the
existence of a subsequence, which is not relabeled, such that
as $(\eps,\tau)\to 0$,
$$
  c_i^{(\tau)}\to c_i\quad\mbox{strongly in }L^2(0,T;H^1(\Omega)),\ i=1,\ldots,n-1.
$$
We deduce from \eqref{3.west}--\eqref{3.ctime} that, possibly for another subsequence,
\begin{align*}
  c_i^{(\tau)} \rightharpoonup c_i &\quad\mbox{weakly in }L^2(0,T;H^2(\Omega)), \\
	\tau^{-1}(c_i^{(\tau)}-\sigma_\tau c_i^{(\tau)}) \rightharpoonup \pa_t c_i
	&\quad\mbox{weakly in }L^2(0,T;H^2(\Omega)'), \\
	w_i^{(\tau)} \rightharpoonup w_i &\quad\mbox{weakly in }L^2(0,T;H^1(\Omega)), \\
	\eps w_i^{(\tau)}\to 0 &\quad\mbox{strongly in }L^2(0,T;H^2(\Omega)),
	\quad i=1,\ldots,n-1.
\end{align*}
We define $c_n:=1-\sum_{i=1}^{n-1}c_i$. Then $c_n^{(\tau)}\to c_n$ strongly
in $L^2(0,T;H^1(\Omega))$ and weakly in $L^2(0,T;H^2(\Omega))$.
Furthermore, $(c_i^{(\tau)})$ converges, up to a subsequence, pointwise a.e.,
and its limit satisfies $\delta\le(\chi_\delta\bm{c})_i\le 1-\delta$, $i=1,\ldots,n$.
The matrix $\widetilde{B}_{ij}^\delta(\widetilde{\bm{c}}^{(\tau)})$ 
is uniformly bounded and
$$
  \widetilde{B}_{ij}^\delta(\widetilde{\bm{c}}^{(\tau)})
	\to \widetilde{B}_{ij}^\delta(\widetilde{\bm{c}})\quad\mbox{strongly in }
	L^q(Q_T)\mbox{ for any } q<\infty,\ i,j=1,\ldots,n.
$$
These convergence results allow us to pass to the limit $(\eps,\tau)\to 0$
in the weak formulation of \eqref{3.tau1}--\eqref{3.tau2} to find that 
$\bm{c}$ solves
$$
  \pa_t c_i = \diver\sum_{j=1}^{n-1}\widetilde{B}_{ij}^\delta(\widetilde{\bm{c}})
	\na w_j, \quad w_i = (h_i^\delta)'(c_i)-(h_n^\delta)'(c_n)
	- \Delta(c_i-c_n)
$$
for $i=1,\ldots,n-1$. Transforming back to the chemical potential $\bm{\mu}$
via $w_i=\mu_i-\mu_n$ and $c_n=1-\sum_{i=1}^{n-1}c_i$, we see that 
$\bm{c}^\delta:=\bm{c}$ solves system \eqref{3.eq}--\eqref{3.bic},
where $\mu_i=(h_i^\delta)'(c_i)-\Delta c_i$.

%%%%%%%%%%%%%%%%%%%%

\subsection{Uniform estimates}\label{sec.unif}

We derive energy and entropy estimates for the solutions to \eqref{3.eq},
being uniform in $\delta$.

\begin{lemma}[Energy and entropy inequalities]\label{lem.eei}
Let $\bm{c}^\delta$ be a weak solution to \eqref{3.eq}--\eqref{3.bic},
constructed in Theorem \ref{thm.approx}. Then the following inequalities hold
for any $T>0$,
\begin{align}
  & \E^\delta(\bm{c}^\delta(\cdot,T)) + \sum_{i,j=1}^n\int_0^T\int_\Omega 
	B_{ij}^\delta(\bm{c}^\delta)\na\mu_i^\delta\cdot\na\mu_j^\delta dxdt 
	\le \E^\delta(\bm{c}^0), \label{3.Ed} \\
	& \H^\delta(\bm{c}^\delta(\cdot,T)) + \sum_{i,j=1}^n\int_0^T\int_\Omega
	B_{ij}^\delta(\bm{c}^\delta)\na(h_i^\delta)'(c_i^\delta)\cdot\na\mu_j^\delta dxdt
	\le \H^\delta(\bm{c}^0), \label{3.Hd} \\
	& \H^\delta(\bm{c}^\delta(\cdot,T)) + \frac{(\evmax-\lambda)^2}{2\evmin\lambda}
	\E^\delta(\bm{c}^\delta(\cdot,T)) + \lambda\sum_{i=1}^n\int_0^T\int_\Omega
	\frac{|\na c_i^\delta|^2}{(\chi_\delta\bm{c}^\delta)_i}dxdt \label{3.EHd} \\
	&\phantom{xxxx}{}+ \lambda\sum_{i=1}^n\int_0^T\int_\Omega(\Delta c_i^\delta)^2 dxdt
	+ \frac{(\evmax-\lambda)^2}{2\lambda}\int_0^T\int_\Omega
	\big|P_L(\chi_\delta\bm{c}^\delta)R(\chi_\delta\bm{c}^\delta)\na\bm{\mu}^\delta
	\big|^2 dxdt \nonumber \\
	&\phantom{xx}{}\le \H^\delta(\bm{c}^0) + \frac{(\evmax-\lambda)^2}{2\evmin\lambda}
	\E^\delta(\bm{c}^0), \nonumber
\end{align}
where $0<\lambda<\evmin$, $\evmin$, $\evmax$ are introduced in
Lemma \ref{lem.DB}, and $R(\chi_\delta\bm{c}^\delta)
=\operatorname{diag}(\sqrt{\chi_\delta\bm{c}^\delta})$.
\end{lemma}

\begin{proof}
Summing \eqref{3.rfei} with $\sigma=1$ over $k=1,\ldots,N$, we find that
\begin{align*}
  \widetilde{\E}^\delta(\widetilde{\bm{c}}^{(\tau)}(\cdot,T))
	&+ \sum_{i,j=1}^{n-1}\int_0^T\int_\Omega \widetilde{B}_{ij}^\delta
	(\widetilde{\bm{c}}^{(\tau)})
	\na w_i^{(\tau)}\cdot\na w_j^{(\tau)} dxdt \\
	&{}+ \eps\sum_{i=1}^n\int_0^T\int_\Omega\big((\Delta w_i^{(\tau)})^2
	+ (w_i^{(\tau)})^2\big)dxdt \le \widetilde{\E}^\delta(\widetilde{\bm{c}}^0).
\end{align*}
We know from \eqref{3.west} and the construction of $\chi_\delta$ that
$(\bm{w}^{(\tau)})$ is bounded in $L^2(0,T;H^1(\Omega))$ and 
$(\widetilde{B}_{ij}^\delta(\widetilde{\bm{c}}))$ is bounded in
$L^\infty(Q_T)$ with respect to $(\eps,\tau)$. 
Therefore, we can pass to the limit $(\eps,\tau)\to 0$ in the previous
inequality, and weak lower semicontinuity of the integral functionals 
leads to \eqref{3.Ed}.

To show \eqref{3.Hd}, we use $(h_i^\delta)'(c_i^\delta)-(h_i^\delta)'(c_n^\delta)$
as a test function in the weak formulation of \eqref{3.cdelta} and sum over
$i=1,\ldots,n-1$:
$$
  \H^\delta(\bm{c}(\cdot,T)) + \sum_{i,j=1}^{n-1}\int_0^T\int_\Omega
	\widetilde{B}_{ij}^\delta(\widetilde{\bm{c}}^\delta)
	\na\big((h_i^\delta)'(c_i^\delta)-(h_i^\delta)'(c_n^\delta)\big)\cdot
	\na w_j^\delta dxdt \le \H^\delta(\bm{c}^0).
$$
This inequality can be rewritten as \eqref{3.Hd} using 
$w_i^\delta=\mu_i^\delta-\mu_n^\delta$. Finally, we derive \eqref{3.EHd}
by combining \eqref{3.Hd} and \eqref{3.Ed} and proceeding as in the proof
of Lemma \ref{lem.fei}.
\end{proof}

\subsection{Proof of Theorem \ref{thm.ex}}\label{sec.exproof}

We perform the limit $\delta\to 0$ to finish the proof of Theorem \ref{thm.ex}. 
It follows from \cite[Lemma 2.1]{ElLu91} that for sufficiently small $\delta>0$, 
there exists $C>0$ (independent of $\delta$)
such that for all $r_1,\ldots,r_n\in\R$ satisfying $\sum_{i=1}^n r_i=1$,
\begin{equation}\label{3.lower}
  \sum_{i=1}^n h_i^\delta(r_i)\ge -C.
\end{equation}
Therefore, estimate \eqref{3.EHd} implies that
\begin{equation}
\label{eqnhelp1}
\begin{aligned}
  \sum_{i=1}^n\int_\Omega|\na c_i^\delta(\cdot,T)|^2 dx
	&+ \sum_{i=1}^n\int_0^T\int_\Omega\frac{|\na c_i^\delta|^2}{
	(\chi_\delta\bm{c}^\delta)_i}dxdt 
	+ \sum_{i=1}^n\int_0^T\int_\Omega(\Delta c_i^\delta)^2 dxdt 
	\\ 
	&{}+ \int_0^T\int_\Omega\big|P_L(\chi_\delta\bm{c}^\delta)
	R(\chi_\delta\bm{c}^\delta)\na\bm{\mu}^\delta\big|^2 dxdt \le C,
\end{aligned}
\end{equation}
and the constant $C>0$ depends on $\evmin$, $\lambda_M$, and $\bm{c}^0$.
Mass conservation (or using the test function $\phi_i=1$ in the weak
formulation of \eqref{3.eq}) shows that
$\int_\Omega c_i^\delta(\cdot,T)dx=\int_\Omega c_0^\delta dx$ for any $T>0$,
i.e.\ $\|\bm{c}^\delta\|_{L^\infty(0,T;L^1(\Omega))}\le C$. We conclude from the
Poincar\'e--Wirtinger inequality that
\begin{equation}\label{3.cd}
  \|\bm{c}^\delta\|_{L^\infty(0,T;H^1(\Omega))}
	+ \|\bm{c}^\delta\|_{L^2(0,T;H^2(\Omega))} \le C.
\end{equation}

Next, we estimate $\pa_t c_i^\delta$. Lemma \ref{lem.Bdelta} implies that
the entries of 
$$(D(\chi_\delta\bm{c}^\delta)P_L(\chi_\delta\bm{c}^\delta)
+ P_{L^\perp}(\chi_\delta\bm{c}^\delta))^{-1}$$
are uniformly bounded. Thus,
by the definition of $D^{BD}(\chi_\delta\bm{c}^\delta)$ and \eqref{2.DBD},
$$
  \int_0^T\int_\Omega\bigg|\sum_{j=1}^n B_{ij}^\delta(\bm{c}^\delta)
	\na\mu_j^\delta\bigg|^2 dxdt \le \lambda_M\int_0^T\int_\Omega
	\big|P_L(\chi_\delta\bm{c}^\delta)R(\chi_\delta\bm{c}^\delta)
	\na\bm{\mu}^\delta\big|^2 dxdt,
$$
and the right-hand side is bounded by \eqref{eqnhelp1}. Setting 
$J_i^\delta:=\sum_{j=1}^n B_{ij}^\delta(\bm{c}^\delta)\na\mu_j^\delta$,
this means that $(J_i^\delta)$ is bounded in $L^2(Q_T)$. Therefore, there exists
a subsequence that is not relabeled such that, as $\delta\to 0$,
$$
  J_i^\delta\rightharpoonup J_i\quad\mbox{weakly in }L^2(Q_T).
$$
This implies that
\begin{equation}
\label{eqnhelp2}
  \|\pa_t c_i^\delta\|_{L^2(0,T;H^1(\Omega)')} \le C.
\end{equation}
We conclude from \eqref{3.cd} and \eqref{eqnhelp2}, using the Aubin--Lions lemma, 
that, for a subsequence (if necessary),
\begin{equation}\label{3.cconv}
\begin{aligned}
  c_i^\delta\to c_i &\quad\mbox{strongly in }L^2(0,T;H^1(\Omega)), \\
	c_i^\delta  \stackrel{\star}{\rightharpoonup} c_i 
	&\quad\mbox{weakly-$\star$\, in }L^\infty(0,T;H^1(\Omega)), \\
	c_i^\delta\rightharpoonup c_i &\quad\mbox{weakly in }L^2(0,T;H^2(\Omega)), \\
	\pa_t c_i^\delta\rightharpoonup \pa_t c_i &\quad\mbox{weakly in }
	L^2(0,T;H^1(\Omega)').
\end{aligned}
\end{equation}
Performing the limit $\delta\to 0$ in \eqref{3.eq}, we see that 
$\pa_t c_i=\diver J_i$ holds in the sense of $L^2(0,T;H^1(\Omega)')$. 

We prove that $c_i\ge 0$ in $Q_T$, $i=1,\ldots,n$,
following \cite{ElLu91}.
By definition \eqref{3.hdelta} and the lower bound \eqref{3.lower}, we have
for $0<\delta<1$,
\begin{align*}
  C&\ge \int_\Omega h_i^\delta(c_i^\delta)dx \ge -C 
	+ \int_{\{c_i^\delta<\delta\}}\bigg(c_i^\delta\log\delta
	- \frac{\delta}{2} + \frac{(c_i^\delta)^2}{2\delta}\bigg)dx \\
	&\ge -C + \int_{\{c_i^\delta<0\}}c_i^\delta\log\delta dx 
	+ \int_{\{0<c_i^\delta<\delta\}}c_i^\delta\log\delta dx - C\delta \\
	&\ge -C + \log\delta\int_{\{c_i^\delta<0\}}c_i^\delta dx
	+ C\delta\log\delta - C\delta.
\end{align*}
Hence, we obtain
$$
  \int_\Omega \max\{0,-c_i^\delta\}dx 
	= \int_{\{c_i^\delta<0\}}|c_i^\delta|dx \le \frac{C}{|\log\delta|}.
$$
The limit $\delta\to 0$ leads to
$$
  \int_\Omega\max\{0,-c_i\}dx \le 0,
$$
implying that $c_i\ge 0$ in $Q_T$. The limit $\delta\to 0$ in
$\sum_{i=1}^n c_i^\delta=1$ gives $\sum_{i=1}^n c_i=1$, hence $c_i\le 1$ 
holds in $Q_T$.

Next, we identify $J_i$ by showing that $J_i=\sum_{j=1}^n 
B_{ij}(\bm{c})\na(\log c_j-\Delta c_j)$ in the sense of distributions.
Inserting the definition of $\mu_i^\delta$ and choosing a test function
$\phi_i\in L^\infty(0,T;$ $W^{2,\infty}(\Omega))$ satisfying $\na\phi_i\cdot\nu=0$
on $\pa\Omega$, we find that
\begin{align}
  \int_0^T&\int_\Omega J_i^\delta\cdot\na\phi_i dxdt
	= \sum_{j=1}^n\int_0^T\int_\Omega B_{ij}^\delta(\bm{c}^\delta)
	\na\phi_i\cdot\na\big((h_j^\delta)'(c_j^\delta)-\Delta c_j^\delta\big) dxdt 
	\nonumber \\
	&= \sum_{j=1}^n\int_0^T\int_\Omega B_{ij}^\delta(\bm{c}^\delta)
	\na\phi_i\cdot\na(h_j^\delta)'(c_j^\delta)dxdt
	+ \sum_{j=1}^n\int_0^T\int_\Omega\Delta c_j^\delta\diver(B_{ij}^\delta(\bm{c}^\delta)
	\na\phi_i)dxdt \label{3.Jdelta} \\
	&=: I_5 + I_6. \nonumber
\end{align}
By definition \eqref{3.Bdelta} of $B_{ij}^\delta(\bm{c}^\delta)$, we have
$$
  I_5 = \sum_{j=1}^n\int_0^T\int_\Omega \sqrt{(\chi_\delta\bm{c}^\delta)_i}
	D_{ij}^{BD}(\chi_\delta\bm{c}^\delta)\na\phi_i\cdot
	\frac{\na c_j^\delta}{\sqrt{(\chi_\delta\bm{c}^\delta)_j}}dxdt.
$$
Lemma \ref{lem.DB} shows that $\sqrt{c_i}D_{ij}^{BD}(\bm{c})/\sqrt{c_j}$ is
bounded in $[0,1]^n$ and in particular when $c_k=0$ for some index $k$.
The strong convergence $\bm{c}^\delta\to \bm{c}$ implies that
$\chi_\delta\bm{c}^\delta\to\bm{c}$ in $L^q(0,T;L^q(\Omega))$ for any $q<\infty$
such that
$$
  I_5\to \sum_{j=1}^n\int_0^T\int_\Omega\sqrt{c_i}D_{ij}^{BD}(\bm{c})
	\frac{1}{\sqrt{c_j}}\na\phi_i\cdot\na c_j dxdt
	= \sum_{j=1}^n\int_0^T\int_\Omega B_{ij}(\bm{c})\na\phi_i\cdot\na\log c_j dxdt.
$$

The limit in $I_6$ is more involved. We decompose $I_6=I_{61}+I_{62}$, where
$$
  I_{61} = \sum_{j=1}^n\int_0^T\int_\Omega\Delta c_j^\delta 
	B_{ij}^\delta(\bm{c}^\delta)\Delta\phi_i dxdt, \quad
	I_{62} = \sum_{j=1}^n\int_0^T\int_\Omega\Delta c_j^\delta
	\na B_{ij}^\delta(\bm{c}^\delta)\cdot\na\phi_i dxdt.
$$
We deduce from the strong convergence of $\bm{c}^\delta$ and the weak
convergence of $\Delta c_j^\delta$ that
$$
  I_{61} \to \sum_{j=1}^n\int_0^T\int_\Omega\Delta c_j B_{ij}(\bm{c})\Delta\phi_i dxdt.
$$
To show the convergence of $I_{62}$, we consider
\begin{align*}
  \int_0^T&\int_\Omega\big|\na\big(B_{ij}^\delta(\bm{c}^\delta)-B_{ij}(\bm{c})\big)
	\big|^2 dxdt \\
	&= \int_0^T\int_\Omega\bigg|\sum_{k=1}^n\bigg\{\bigg(
	\frac{\pa B_{ij}^\delta}{\pa c_k}(\bm{c}^\delta) - \frac{\pa B_{ij}}{\pa c_k}(\bm{c})
	\bigg)\na c_k + \frac{\pa B_{ij}^\delta}{\pa c_k}(\bm{c}^\delta)
	\na(c_k^\delta-c_k)\bigg\}\bigg|^2 dxdt.
\end{align*}
By Lemma \ref{lem.DB} (i), $\pa D_{ij}^{BD}/\pa c_k$ exists and is bounded
in $[0,1]^n$. Then, by the definition of $B_{ij}(\bm{c})$, we have
$(\pa B_{ij}^\delta/\pa c_k)(\bm{c}^\delta)\to (\pa B_{ij}/\pa c_k)(\bm{c})$
strongly in $L^2(Q_T)$. It follows from $\na c_k^\delta\to \na c_k$ strongly 
in $L^2(Q_T)$ that the right-hand side of the previous identity
converges to zero. We infer that 
$$
  I_{62}\to \sum_{j=1}^n\int_0^T\int_\Omega\Delta c_j \na B_{ij}(\bm{c})\cdot
	\na\phi_i dxdt.
$$
Consequently, we have
\begin{align*}
  I_6 \to  &\sum_{j=1}^n\int_0^T\int_\Omega\Delta c_j\big(B_{ij}(\bm{c})\Delta\phi_i
	+  \na B_{ij}(\bm{c})\cdot\na\phi_i\big) dxdt \\
  &= \sum_{j=1}^n\int_0^T\int_\Omega\Delta c_j\diver(B_{ij}(\bm{c})\na\phi_i)dxdt.
\end{align*}

We have shown that \eqref{3.Jdelta} becomes in the limit $\delta\to 0$
$$
  \int_0^T\int_\Omega J_i\cdot\na\phi dxdt
	= \sum_{j=1}^n\int_0^T\int_\Omega\big(B_{ij}(\bm{c})\na\phi_i\cdot\na\log c_j
	+ \Delta c_j\diver(B_{ij}(\bm{c})\na\phi_i)\big)dxdt
$$
and hence, in the sense of distributions,
$$
  J_i = \sum_{j=1}^n B_{ij}(\bm{c})\na(\log c_j-\Delta c_j), \quad i=1,\ldots,n.
$$

{\em Step 2: Energy and entropy inequalities.}
The limit $c_i^\delta\rightharpoonup c_i$ weakly-$\star$ in $L^\infty(0,T;H^1(\Omega))$
(see \eqref{3.cconv}) and the weak lower semicontinuity of the energy and 
entropy show that
$$
  \H(\bm{c}(\cdot,T))\le \liminf_{\delta\to 0}\H^\delta(\bm{c}^\delta(\cdot,T)), \quad
	\E(\bm{c}(\cdot,T))\le \liminf_{\delta\to 0}\E^\delta(\bm{c}^\delta(\cdot,T)).
$$
Moreover, because of the weak convergence of $\Delta c_i^\delta$ 
in $L^2(Q_T)$ from \eqref{3.cconv},
$$
  \sum_{i=1}^n\int_0^T\int_\Omega(\Delta c_i)^2 dxdt
	\le \liminf_{\delta\to 0}\sum_{i=1}^n\int_0^T\int_\Omega(\Delta c_i^\delta)^2 dxdt.
$$
The combined energy-entropy inequality \eqref{3.EHd} and the property
$|\na(\chi_\delta\bm{c}^\delta)_i|\le|\na c_i^\delta|$ give
$$
  \big\|\na\sqrt{(\chi_\delta\bm{c}^\delta)_i}\big\|_{L^2(Q_T)}
	= \frac12\bigg\|\frac{\na c_i^\delta}{\sqrt{(\chi_\delta\bm{c}^\delta)_i}}
	\bigg\|_{L^2(Q_T)} \le C,
$$
which, together with $(\chi_\delta\bm{c}^\delta)_i\to c_i$ strongly in
$L^2(Q_T)$ leads to
\begin{equation}\label{3.sqrtc}
  \na\sqrt{(\chi_\delta\bm{c}^\delta)_i}\rightharpoonup\na\sqrt{c_i}
	\quad\mbox{weakly in }L^2(Q_T).
\end{equation}
We conclude that
$$
  \|\na\sqrt{c_i}\|_{L^2(Q_T)}\le \liminf_{\delta\to 0}
	\big\|\na\sqrt{(\chi_\delta\bm{c}^\delta)_i}\big\|_{L^2(Q_T)}.
$$

Finally, by \eqref{3.EHd}, we observe that $P_L(\chi_\delta\bm{c}^\delta)
R(\chi_\delta\bm{c}^\delta)\na\bm{\mu}^\delta$ is uniformly bounded
in $L^2(Q_T)$ such that, up to a subsequence,
$$
  P_L(\chi_\delta\bm{c}^\delta)R(\chi_\delta\bm{c}^\delta)\na\bm{\mu}^\delta
  \rightharpoonup \bm{\zeta}\quad\mbox{weakly in }L^2(Q_T).
$$
Hence, again by weak lower semicontinuity of the norm,
$$
  \|\bm{\zeta}\|_{L^2(0,T;L^2(\Omega))}
	\le\liminf_{\delta\to 0}\big\|(P_L(\chi_\delta\bm{c}^\delta)
	R(\chi_\delta\bm{c}^\delta)\na\bm{\mu}^\delta\big\|_{L^2(0,T;L^2(\Omega))}.
$$
It remains to take the limit inferior $\delta\to 0$ in \eqref{3.EHd}
to conclude that the combined energy-entropy inequality \eqref{1.EH} holds.

\begin{lemma}[Identification of $\bm{\zeta}$]\label{lem.ident}
Let \eqref{1.cregul} hold and let $\bm{\zeta}$ be the weak $L^2(Q_T)$ limit
of $P_L(\chi_\delta\bm{c}^\delta)R(\chi_\delta\bm{c}^\delta)\na\bm{\mu}^\delta$.
Then $\bm{\zeta}=P_L(\bm{c})R(\bm{c})\na\bm{\mu}$.
\end{lemma}

\begin{proof}
Let $\phi_i\in C_0^\infty(Q_T)$ be a test function. Then, inserting the definition
$\mu_j^\delta=(h_j^\delta)'(c_j^\delta)-\Delta c_j^\delta$ and integrating by parts,
\begin{align}
  \sum_{j=1}^n&\int_0^T\int_\Omega\Big(P_L(\chi_\delta\bm{c}^\delta)_{ij}
	\sqrt{(\chi_\delta\bm{c}^\delta)_j}\na\mu_j^\delta
	- P_L(\bm{c})_{ij}\sqrt{c_j}\na\mu_j\Big)\cdot\na\phi_i dxdt \nonumber \\
	&= \sum_{j=1}^n\int_0^T\int_\Omega\Big(P_L(\chi_\delta\bm{c}^\delta)_{ij}
	\sqrt{(\chi_\delta\bm{c}^\delta)_j}\na (h_j^\delta)'(c_j^\delta)
	- P_L(\bm{c})_{ij}\sqrt{c_j}\na\log c_j\Big)\cdot\na\phi_i dxdt \label{4.id} \\
	&\phantom{xx}{}+ \sum_{j=1}^n\int_0^T\int_\Omega\diver\Big\{
	\Big(P_L(\chi_\delta\bm{c}^\delta)_{ij}
	\sqrt{(\chi_\delta\bm{c}^\delta)_j}-P_L(\bm{c})_{ij}\sqrt{c_j}\Big)\na\phi_i\Big\}
	\Delta c^\delta_j dxdt \nonumber \\
	&\phantom{xx}{}+ \sum_{j=1}^n\int_0^T\int_\Omega\diver\big(P_L(\bm{c})_{ij}\sqrt{c_j}
	\na\phi_i\big)\Delta(c_j^\delta-c_j)dxdt. \nonumber
\end{align}
The bracket in the first integral on the right-hand side can be written as
\begin{align*}
  P_L(&\chi_\delta\bm{c}^\delta)_{ij}
	\sqrt{(\chi_\delta\bm{c}^\delta)_j}\na (h_j^\delta)'(c_j^\delta)
	- P_L(\bm{c})_{ij}\sqrt{c_j}\na\log c_j \\
	&= P_L(\chi_\delta\bm{c}^\delta)_{ij}
	\frac{\na c_j^\delta}{\sqrt{(\chi_\delta\bm{c}^\delta)_j}}
	- P_L(\bm{c})_{ij}\frac{\na c_j}{\sqrt{c_j}}.
\end{align*}
Thanks to the convergences \eqref{3.cconv} and \eqref{3.sqrtc}, we can pass to
the limit $\delta\to 0$ in \eqref{4.id}:
$$
  \lim_{\delta\to 0}\sum_{j=1}^n\int_0^T\int_\Omega
	\Big(P_L(\chi_\delta\bm{c}^\delta)_{ij}
	\sqrt{(\chi_\delta\bm{c}^\delta)_j}\na\mu_j^\delta
	- P_L(\bm{c})_{ij}\sqrt{c_j}\na\mu_j\Big)\cdot\na\phi_i dxdt = 0.
$$
By the uniqueness of the limit, the claim
$\bm{\zeta}=P_L(\bm{c})R(\bm{c})\na\bm{\mu}$ follows. 
\end{proof}

%%%%%%%%%%%%%%%%%%%%%%%%%%%%%%%%%%%%%%%%%%%%%%%%%%%%%%%%%%%%%%%%%%%%%%%%%%%%%%

\section{Proof of Theorem \ref{thm.wsu}}\label{sec.wsu}

In this section, we prove the weak-strong uniqueness property.
First, we compute a combined {\em relative} energy-entropy inequality.
Then we use this inequality to derive a stability estimate, which leads
to the desired weak-strong uniqueness result.

\subsection{Evolution of the relative energy and entropy}\label{sec.wsu.formal}

We start by calculating the time evolution of the relative entropy \eqref{1.relH}
and the relative energy \eqref{1.relE} for {\em smooth} solutions $\bm{c}$ and
$\bar{\bm{c}}$. Inserting \eqref{2.B} and integrating by parts leads to
\begin{align*}
  \frac{d\H}{dt}(\bm{c}|\bar{\bm{c}})
	&= \sum_{i=1}^n\int_\Omega\bigg(\log\frac{c_i}{\bar{c}_i}\pa_t c_i
	- \bigg(\frac{c_i}{\bar{c}_i}-1\bigg)\pa_t\bar{c}_i\bigg)dx \\
	&= -\sum_{i,j=1}^n\int_\Omega B_{ij}(\bm{c})
	\na\log\frac{c_i}{\bar{c}_i}\cdot\na\mu_j dx
	+ \sum_{i,j=1}^n\int_\Omega B_{ij}(\bar{\bm{c}})\na\bigg(\frac{c_i}{\bar{c}_i}
	\bigg)\cdot\na\bar{\mu}_j dx \\
	&= -\sum_{i,j=1}^n\int_\Omega B_{ij}(\bm{c})\na\log\frac{c_i}{\bar{c}_i}
	\cdot\na(\mu_j-\bar{\mu}_j)dx \\
	&\phantom{xx}{}
	- \sum_{i,j=1}^n\int_\Omega\bigg(B_{ij}(\bm{c}) - \frac{c_i}{\bar{c}_i}
	B_{ij}(\bar{\bm{c}})\bigg)\na\log\frac{c_i}{\bar{c}_i}\cdot\na\bar{\mu}_j dx.
\end{align*}

Next, we compute
\begin{align}
  \frac{d\E}{dt}(\bm{c}|\bar{\bm{c}}) 
	&= \sum_{i=1}^n\int_\Omega\bigg(\log\frac{c_i}{\bar{c}_i}\pa_t c_i
	- \bigg(\frac{c_i}{\bar{c}_i}-1\bigg)\pa_t\bar{c}_i\bigg)dx
	+ \sum_{i=1}^n\int_\Omega\na(c_i-\bar{c}_i)\cdot\na\pa_t(c_i-\bar{c}_i)dx \nonumber \\
	&= \sum_{i=1}^n\bigg\{\bigg(\log\frac{c_i}{\bar{c}_i} - \Delta(c_i-\bar{c}_i)\bigg)
	\pa_t c_i - \bigg(\frac{c_i}{\bar{c}_i}-1-\Delta(c_i-\bar{c}_i)\bigg)\pa_t\bar{c}_i
	\bigg\}dx \label{4.aux0} \\
	&= -\sum_{i,j=1}^n\int_\Omega B_{ij}(\bm{c})\na(\mu_i-\bar{\mu}_i)\cdot\na\mu_j dx 
	\nonumber \\
	&\phantom{xx}{}+ \sum_{i,j=1}^n\int_\Omega B_{ij}(\bar{\bm{c}})
	\na\bigg(\frac{c_i}{\bar{c}_i}-1-\Delta(c_i-\bar{c}_i)\bigg)\cdot
	\na\bar{\mu}_j dx. \nonumber
\end{align}
We add and subtract the expression $\sum_{i=1}^n\int_\Omega
B_{ij}(\bm{c})\na(\mu_i-\bar{\mu}_i)\cdot\na\bar{\mu}_j dx$:
\begin{align}
  \frac{d\E}{dt}(\bm{c}|\bar{\bm{c}}) 
	&= -\sum_{i=1}^n\int_\Omega B_{ij}(\bm{c})\na(\mu_i-\bar{\mu}_i)
	\cdot\na(\mu_j-\bar{\mu}_j) dx \nonumber \\
	&\phantom{xx}{}+ \sum_{i,j=1}^n\int_\Omega\bigg\{B_{ij}(\bar{\bm{c}})
	\bigg(\frac{c_i}{\bar{c}_i}\na\log\frac{c_i}{\bar{c}_i} - \na\Delta(c_i-\bar{c}_i)
	\bigg) - B_{ij}(\bm{c})\na(\mu_i-\bar{\mu}_i)\bigg\}\cdot\na\bar{\mu}_j dx 
	\label{4.aux} \\
	&= -\sum_{i,j=1}^n\int_\Omega B_{ij}(\bm{c})\na(\mu_i-\bar{\mu}_i)
	\cdot\na(\mu_j-\bar{\mu}_j) dx \nonumber \\
	&\phantom{xx}{}- \sum_{i,j=1}^n\int_\Omega\bigg(B_{ij}(\bm{c})
	- \frac{c_i}{\bar{c}_i}B_{ij}(\bar{\bm{c}})\bigg)\na(\mu_i-\bar{\mu}_i)\cdot
	\na\bar{\mu}_j dx \nonumber \\
	&\phantom{xx}{}+ \sum_{i,j=1}^n\int_\Omega B_{ij}(\bar{\bm{c}})
	\bigg(\frac{c_i}{\bar{c}_i}-1\bigg)\na\Delta(c_i-\bar{c}_i)\cdot\na\bar{\mu}_j dx.
	\nonumber 
\end{align}
We want to reformulate the expression $\bar{c}_i^{-1}(c_i-\bar{c}_i)
\na\Delta(c_i-\bar{c}_i)$ in the last integral. 
For this, we observe that for any smooth function $f$, it holds that
\begin{align*}
  f\na\Delta f &= \na(f\Delta f) - \na f\Delta f
	= \na\big(\diver(f\na f)- |\na f|^2\big)
	- \diver(\na f\otimes\na f) + \frac12\na|\na f|^2 \\
	&= \na\diver(f\na f) - \frac12\na|\na f|^2 - \diver(\na f\otimes\na f).
\end{align*}
Therefore,
\begin{align*}
  (c_i-\bar{c}_i)\na\Delta(c_i-\bar{c}_i)
	&= \na\diver\big((c_i-\bar{c}_i)\na(c_i-\bar{c}_i)\big) 
	- \frac12\na|\na(c_i-\bar{c}_i)|^2 \\
	&\phantom{xx}{}- \diver\big(\na(c_i-\bar{c}_i)\otimes\na(c_i-\bar{c}_i)\big).
\end{align*}
Inserting this expression into the last term of \eqref{4.aux} and integrating
by parts, we find that
\begin{align*}
  \frac{d\E}{dt}(\bm{c}|\bar{\bm{c}}) 
	&= -\sum_{i=1}^n\int_\Omega B_{ij}(\bm{c})\na(\mu_i-\bar{\mu}_i)
	\cdot\na(\mu_j-\bar{\mu}_j) dx \\
	&\phantom{xx}{}- \sum_{i,j=1}^n\int_\Omega\bigg(B_{ij}(\bm{c})
	- \frac{c_i}{\bar{c}_i}B_{ij}(\bar{\bm{c}})\bigg)\na(\mu_i-\bar{\mu}_i)\cdot
	\na\bar{\mu}_j dx \\
	&\phantom{xx}{}+ \sum_{i,j=1}^n\int_\Omega (c_i-\bar{c}_i)\na(c_i-\bar{c}_i)
	\cdot\na\diver\bigg(\frac{1}{\bar{c}_i}B_{ij}(\bar{\bm{c}})\na\bar{\mu}_j\bigg)dx \\
	&\phantom{xx}{}+ \frac12\sum_{i,j=1}^n\int_\Omega|\na(c_i-\bar{c}_i)|^2
	\diver\bigg(\frac{1}{\bar{c}_i}B_{ij}(\bar{\bm{c}})\na\bar{\mu}_j\bigg)dx \\
	&\phantom{xx}{}+ \sum_{i,j=1}^n\int_\Omega
	\na(c_i-\bar{c}_i)\otimes\na(c_i-\bar{c}_i)
	:\na\otimes\bigg(\frac{1}{\bar{c}_i}B_{ij}(\bar{\bm{c}})\na\bar{\mu}_j\bigg)dx,
\end{align*}
where $\na\otimes(\bar{c}_i^{-1}B_{ij}(\bar{\bm{c}})\na\bar{\mu}_j)$ is
a matrix with entries $\pa_{x_k}(\bar{c}_i^{-1}B_{ij}(\bar{\bm{c}})
\pa_{x_\ell}\bar{\mu}_j)$
for $k,\ell=1,\ldots,n$ and ``:'' denotes the Frobenius matrix product.

The following lemma states the rigorous result. Since we suppose that 
the weak solution satisfies energy and entropy {\em inequalities}
instead of {\em equalities}, we obtain also inequalities for the relative
energy and entropy.

\begin{lemma}[Relative energy and entropy]%\label{lem.rel}
Let $\bm{c}$ and $\bar{\bm{c}}$ be a weak and strong solution to
\eqref{1.eq1}--\eqref{1.mu} with initial data $\bm{c}^0$ and $\bar{\bm{c}}^0$,
respectively. Assume that $\bm{c}$ satisfies the regularity \eqref{1.cregul}
and the energy and entropy inequalites \eqref{1.dEdt}--\eqref{1.dHdt}.
Furthermore, we suppose that $\bar{\bm{c}}$ is strictly positive and
satisfies the regularity
$$
  \bar{\mu}_i = \log\bar{c}_i-\Delta\bar{c}_i\in L_{\rm loc}^2(0,\infty;H^2(\Omega)), 
	\quad \bar{c}_i\in L_{\rm loc}^\infty(0,\infty;W^{3,\infty}(\Omega)),
	\quad i=1,\ldots,n.
$$
Then the following relative energy and entropy inequalities hold for any $T>0$:
\begin{align}\label{4.relE}
  \E(&\bm{c}(T)|\bar{\bm{c}}(T))
	+ \sum_{i=1}^n\int_0^T\int_\Omega B_{ij}(\bm{c})\na(\mu_i-\bar{\mu}_i)
	\cdot\na(\mu_j-\bar{\mu}_j)dxdt \\
	&\le \E(\bm{c}^0|\bar{\bm{c}}^0)
	- \sum_{i,j=1}^n\int_0^T\int_\Omega\bigg(B_{ij}(\bm{c})
	- \frac{c_i}{\bar{c}_i}B_{ij}(\bar{\bm{c}})\bigg)\na(\mu_i-\bar{\mu}_i)\cdot
	\na\bar{\mu}_j dxdt \nonumber \\
	&\phantom{xx}{}+ \sum_{i,j=1}^n\int_0^T\int_\Omega (c_i-\bar{c}_i)\na(c_i-\bar{c}_i)
	\cdot\na\diver\bigg(\frac{1}{\bar{c}_i}B_{ij}(\bar{\bm{c}})\na\bar{\mu}_j\bigg)dxdt 
	\nonumber \\
	&\phantom{xx}{}+ \frac12\sum_{i,j=1}^n\int_0^T\int_\Omega|\na(c_i-\bar{c}_i)|^2
	\diver\bigg(\frac{1}{\bar{c}_i}B_{ij}(\bar{\bm{c}})\na\bar{\mu}_j\bigg)dxdt 
	\nonumber \\
	&\phantom{xx}{}+ \sum_{i,j=1}^n\int_0^T\int_\Omega
	\na(c_i-\bar{c}_i)\otimes\na(c_i-\bar{c}_i)
	:\na\otimes\bigg(\frac{1}{\bar{c}_i}B_{ij}(\bar{\bm{c}})\na\bar{\mu}_j\bigg)dxdt, 
	\nonumber \\
	\H(&\bm{c}(T)|\bar{\bm{c}}(T)) \le \H(\bm{c}^0|\bar{\bm{c}}^0)
	- \sum_{i,j=1}^n\int_0^T\int_\Omega B_{ij}(\bm{c})\na\log\frac{c_i}{\bar{c}_i}
	\cdot\na(\mu_j-\bar{\mu}_j)dxdt \label{4.relH} \\
	&\phantom{xx}{}- \sum_{i,j=1}^n\int_0^T\int_\Omega
  \bigg(B_{ij}(\bm{c}) - \frac{c_i}{\bar{c}_i}B_{ij}(\bar{\bm{c}})\bigg)
	\na\log\frac{c_i}{\bar{c}_i}\cdot\na\bar{\mu}_j dxdt. \nonumber
\end{align}
\end{lemma}

The integrals in \eqref{4.relE} and \eqref{4.relH} are well defined 
because of the regularity
properties for weak solutions $\bm{c}$ and the regularity assumptions 
on the strong solution $\bar{\bm{c}}$.
Indeed, we have $B_{ij}(\bm{c})\na\mu_j\in L^2(Q_T)$ (see \eqref{1.regflux}),
$B_{ij}(\bm{c})\na\log c_i=2D_{ij}^{BD}(\bm{c})\sqrt{c_j}\na\sqrt{c_i}
\in L^2(Q_T)$ (see \eqref{1.EH}), and using the definition \eqref{1.B}, 
the assumption \eqref{1.cregul}, and Lemma \ref{lem.DB} (i), we have
$$
  B_{ij}(\bm{c})\na\mu_i\cdot\na\mu_j 
	= D_{ij}^{BD}(\bm{c})\big(2\na\sqrt{c_i} - \sqrt{c_i}\na\Delta c_i\big)
	\cdot\big(2\na\sqrt{c_j} - \sqrt{c_j}\na\Delta c_j\big)\in L^1(Q_T).
$$

\begin{proof}
The relative energy and entropy inequalities are proved from the weak
formulation of \eqref{1.eq1} by choosing suitable test functions. For this,
we observe that, by \eqref{1.weak}, $c_i-\bar{c}_i$ satisfies 
\begin{align}
  0 &= \int_0^\infty\int_\Omega(c_i-\bar{c}_i)\pa_t\phi_i dxdt  
	+ \int_\Omega ( c_i^0 (x) - \bar c_i^0 (x) ) \phi_i (x,0) dx  \label{4.start} \\
	&\phantom{xx}{} - \sum_{j=1}^n\int_0^\infty\int_\Omega\big(B_{ij}(\bm{c})\na\log c_j
	- B_{ij}(\bar{\bm{c}})\na\log\bar{c}_j\big)\cdot\na\phi_i dxdt \nonumber \\
	&\phantom{xx}{}- \sum_{j=1}^n\int_0^\infty\int_\Omega
	\Big(\diver \big (B_{ij}(\bm{c})\na\phi_i \big) \Delta c_j 
	- \diver \big ( B_{ij}(\bar{c})\na\phi_i \big )
	\Delta\bar{c}_j\Big)dxdt. \nonumber
\end{align}
By density, this formulation also holds for $\phi_i=\bar{\mu}_i\theta_\eps(t)$, where
$$
  \theta_\eps(t) = \left\{\begin{array}{ll}
	1 &\quad\mbox{for }0\le t\le T, \\
	(T-t)/\eps + 1 &\quad\mbox{for }T<t<T+\eps, \\
	0 &\quad\mbox{for }t\ge T+\eps.
	\end{array}\right.
$$
Then, passing to the limit $\eps\to 0$ and summing over $i=1,\ldots,n$, we arrive at
\begin{align*}
  \sum_{i=1}^n&\int_\Omega(c_i-\bar{c}_i)\bar{\mu}_i dx\Bigg|_0^T
	= \sum_{i=1}^n\int_0^T\langle \pa_t\bar{\mu}_i,c_i-\bar{c}_i\rangle dt \\
	&\phantom{xx}{}- \sum_{i,j=1}^n\int_0^T\int_\Omega\big(B_{ij}(\bm{c})
	\na\log c_j\cdot\na\bar{\mu}_i
	+ \diver(B_{ij}(\bm{c})\na\bar{\mu}_i)\Delta c_j\big)dxdt \\
	&\phantom{xx}{}+\sum_{i,j=1}^n\int_0^T\int_\Omega\big(B_{ij}(\bar{\bm{c}})
	\na\log\bar{c}_j\cdot\na\bar{\mu}_i
	+ \diver(B_{ij}(\bar{\bm{c}})\na\bar{\mu}_i)\Delta\bar{c}_j\big)dxdt \\
	&=: I_7 + I_8 + I_9,
\end{align*}
where $\langle\cdot,\cdot\rangle$ is the duality bracket between $H^1(\Omega)'$ and
$H^1(\Omega)$. This product is well defined ,since it holds 
in the sense of $H^1(\Omega)'$ that
$$
  \pa_t\bar{\mu}_i = \pa_t(\log\bar{c}_i - \Delta\bar{c}_i)
	= \sum_{j=1}^n\frac{1}{\bar{c}_i}\diver(B_{ij}(\bar{\bm{c}})\na\bar{\mu}_j)
	- \sum_{j=1}^n\Delta\diver(B_{ij}(\bar{\bm{c}})\na\bar{\mu}_j).
$$
Inserting this expression into $I_7$, the dual
product can be written as an integral:
\begin{align*}
  I_7 &= -\sum_{i,j=1}^n\int_0^T\int_\Omega\bigg(B_{ij}(\bar{\bm{c}})
	\na\bigg(\frac{c_i}{\bar{c}_i}-1\bigg)\cdot\na\bar{\mu}_j
	+ \Delta(c_i-\bar{c}_i)\diver(B_{ij}(\bar{\bm{c}})\na\bar{\mu}_j)\bigg)dxdt \\
	&= -\sum_{i,j=1}^n\int_0^T\int_\Omega B_{ij}(\bar{\bm{c}})
	\na\bigg(\frac{c_i}{\bar{c}_i}-1\bigg)\cdot\na\bar{\mu}_j dxdt \\
	&\phantom{xx}{} - \sum_{i,j=1}^n\int_0^T\int_\Omega\bar{c}_i\Delta(c_i-\bar{c}_i)
	\diver\bigg(\frac{1}{\bar{c}_i}B_{ij}(\bar{\bm{c}})\na\bar{\mu}_j\bigg)dxdt \\
	&\phantom{xx}{}- \sum_{i,j=1}^n\int_0^T\int_\Omega\frac{1}{\bar{c}_i}
	B_{ij}(\bar{\bm{c}})\Delta(c_i-\bar{c}_i)\na\bar{c}_i\cdot\na\bar{\mu}_j dxdt.
\end{align*}
Replacing $\Delta c_j$ by $\log c_j-\mu_j$ in $I_8$ and integrating by parts 
in the term involving the divergence, some terms cancel and we find that
\begin{align*}
  I_8 &= -\sum_{i,j=1}^n\int_0^T\int_\Omega\big( B_{ij}(\bm{c})
	\na\bar{\mu}_i\cdot\na\log c_j + \diver(B_{ij}(\bm{c})\na\bar{\mu}_i)
	(\log c_j-\mu_j)\big)dxdt \\
	&=-\sum_{i,j=1}^n\int_0^T\int_\Omega B_{ij}(\bm{c})\na\bar{\mu}_i\cdot\na\mu_j dxdt.
\end{align*}
Assumption \eqref{1.cregul} guarantees that the flux
has the regularity $\sum_{j=1}^n B_{ij}(\bm{c})\na\mu_j\in L^2(Q_T)$
such that the last integral is defined.
The remaining term $I_9$ is reformulated in a similar way, leading to
$$
  I_9 = \sum_{i,j=1}^n\int_0^T\int_\Omega B_{ij}(\bar{\bm{c}})\na\bar{\mu}_i
	\cdot\na\bar{\mu}_j dxdt.
$$
It follows from the definition of the relative energy, the inequality \eqref{1.dEdt}
for $\E(\bm{c})$, and the identity \eqref{1.dEdtbar} for $\E(\bar{\bm{c}})$ that
\begin{align*}
  \E(&\bm{c}(T)|\bar{\bm{c}}(T)) - \E(\bm{c}^0|\bar{\bm{c}}^0) \\
	&= \big(\E(\bm{c}(T))-\E(\bm{c}^0)\big)
	- \big(\E(\bar{\bm{c}}(T))-\E(\bar{\bm{c}}^0)\big)
	- \int_\Omega\bar{\bm{\mu}}\cdot(\bm{c}-\bar{\bm{c}})dx \Big|_0^T \\
	&\le -\sum_{i,j=1}^n\int_0^T\int_\Omega\big(B_{ij}(\bm{c})\na\mu_i\cdot\na\mu_j
	- B_{ij}(\bar{\bm{c}})\na\bar{\mu}_i\cdot\na\bar{\mu}_j\big)dxdt   
	- (I_7 + I_8 + I_9) \\
	&= -\sum_{i,j=1}^n\int_0^T\int_\Omega B_{ij}(\bm{c})\na(\mu_i-\bar{\mu}_i)
	\cdot\na\mu_j dxdt \\
	&\phantom{xx}{}-\sum_{i,j=1}^n\int_0^T\int_\Omega B_{ij}(\bar{\bm{c}})
	\na\bigg(\frac{c_i}{\bar{c}_i}-1\bigg)\cdot\na\bar{\mu}_j dxdt \\
	&\phantom{xx}{} - \sum_{i,j=1}^n\int_0^T\int_\Omega\bar{c}_i\Delta(c_i-\bar{c}_i)
	\diver\bigg(\frac{1}{\bar{c}_i}B_{ij}(\bar{\bm{c}})\na\bar{\mu}_j\bigg)dxdt \\
	&\phantom{xx}{}- \sum_{i,j=1}^n\int_0^T\int_\Omega\frac{1}{\bar{c}_i}
	B_{ij}(\bar{\bm{c}})\Delta(c_i-\bar{c}_i)\na\bar{c}_i\cdot\na\bar{\mu}_j dxdt.
\end{align*}
This inequality is just a reformulation of \eqref{4.aux0}, which leads,
by proceeding as in \eqref{4.aux} and the subsequent calculations, to \eqref{4.relE}.

Next, we verify the relative entropy inequality. Taking the test function
$\phi_i=(\log\bar{c}_i)\theta_\eps (t)$ in \eqref{4.start}, passing to the limit
$\eps\to 0$, and summing over $i=1,\ldots,n$ leads to
\begin{align*}
  \sum_{i=1}^n&\int_\Omega(c_i-\bar{c}_i)\log\bar{c}_i dx\bigg|_0^T
	= \sum_{i=1}^n\int_0^T \int_\Omega(c_i-\bar{c}_i)\pa_t(\log\bar{c}_i)dxdt \\
	&\phantom{xx}{} - \sum_{j=1}^n\int_0^\infty\int_\Omega\big(B_{ij}(\bm{c})\na\log c_j
	- B_{ij}(\bar{\bm{c}})\na\log\bar{c}_j\big)\cdot\na \log \bar c_i dxdt \nonumber \\
	&\phantom{xx}{}- \sum_{j=1}^n\int_0^\infty\int_\Omega
	\Big(\diver \big (B_{ij}(\bm{c})\na \log \bar c_i  \big) \Delta c_j 
	- \diver \big ( B_{ij}(\bar{c})\na  \log \bar c_i   \big )
	\Delta\bar{c}_j\Big)dxdt. \nonumber
\end{align*}
This yields, together with \eqref{1.dHdt}, \eqref{1.dHdtbar}, an integration by 
parts, and regularity assumption \eqref{1.cregul}, that
\begin{align*}
  \H(&\bm{c}(T)|\bar{\bm{c}}(T)) - \H(\bm{c}^0|\bar{\bm{c}}^0) \\
	&= \big(\H(\bm{c}(T))-\H(\bm{c}^0)\big) - \big(\H(\bar{\bm{c}}(T))-\H(\bar{\bm{c}}^0)
	\big) - \int_\Omega(\bm{c}-\bar{\bm{c}})\cdot \log\bar{\bm{c}}\,dx\bigg|_0^T 
	\\
  &\le -\sum_{i,j=1}^n\int_0^T\int_\Omega  \Big ( B_{ij}(\bm{c}) \na \log c_i 
	\cdot  \na \mu_j -  B_{i j}(\bar{\bm{c}})  \na \log \bar c_i  \cdot 
	\nabla \bar \mu_j  \Big ) dx dt \\
  &\phantom{xx}{} - \sum_{i=1}^n\int_0^T \int_\Omega(c_i-\bar{c}_i)
	\pa_t(\log\bar{c}_i)dxdt \\
  &\phantom{xx}{} +  \sum_{i,j=1}^n\int_0^\infty\int_\Omega
	\Big (  B_{ij}(\bm{c})  \na \mu_j \cdot \na \log \bar c_i   
	- B_{ij}(\bar{\bm{c}}) \na \bar \mu_j \cdot  \na  \log \bar c_i   \big )
	\Big)dxdt. \nonumber \\
 &=  -\sum_{i,j=1}^n\int_0^T\int_\Omega 
	\Big ( B_{ij}(\bm{c}) \na\mu_j \cdot \nabla \bigg ( \log \frac{c_i}{\bar c_i} \bigg )
	-  \nabla \bigg ( \frac{c_i}{\bar c_i} - 1 \bigg ) \cdot B_{i j}(\bar{\bm{c}}) 
	\nabla \bar \mu_j \Big ) dx dt,
\end{align*}
which readily gives \eqref{4.relH}.
\end{proof}

%%%%%%%%%%%%%%%%%%%%%%%

\subsection{Proof of the weak-strong uniqueness property}\label{sec.wsu.rig}

We proceed with the proof of Theorem \ref{thm.wsu}. First, we estimate the
relative entropy inequality \eqref{4.relH} and then the relative energy
inequality \eqref{4.relE}. A combination of both estimates shows \eqref{1.comb},
which proves the weak-strong uniqueness property.

{\em Step 1: Estimating the relative entropy.} As in the proof of Lemma \ref{lem.fei},
we decompose the matrix $B(\bm{c})$ by setting $M(\bm{c}):=B(\bm{c})-\lambda G(\bm{c})$
such that $B(\bm{c}) = M(\bm{c}) + \lambda G(\bm{c})$, where 
$G(\bm{c})=R(\bm{c})P_L(\bm{c})R(\bm{c})$ has the entries
$G_{ij}(\bm{c}) = c_i\delta_{ij}-c_ic_j$ and $0<\lambda<\evmin$.
In terms of these matrices, we can formulate \eqref{4.relH} as
\begin{align}\label{4.HGM}
  \H(&\bm{c}(T)|\bar{\bm{c}}(T)) - \H(\bm{c}^0|\bar{\bm{c}}^0)
	\le -\sum_{i,j=1}^n\int_0^T\int_\Omega M_{ij}(\bm{c})\na\log\frac{c_i}{\bar{c}_i}\cdot
	\na(\mu_j-\bar{\mu}_j)dxdt \\
	&{}- \lambda\sum_{i,j=1}^n\int_0^T\int_\Omega G_{ij}(\bm{c})
	\na\log\frac{c_i}{\bar{c}_i}\cdot\na(\mu_j-\bar{\mu}_j)dxdt \nonumber \\
	&{}- \sum_{i,j=1}^n\int_0^T\int_\Omega\bigg(B_{ij}(\bm{c})
	- \frac{c_i}{\bar{c}_i}B_{ij}(\bar{\bm{c}})\bigg)
	\na\log\frac{c_i}{\bar{c}_i}\cdot\na\bar{\mu}_j dxdt
	=: I_{10} + I_{11} + I_{12}. \nonumber
\end{align}

{\em Step 1a: Estimate of $I_{10}$.}
We know from \eqref{psd.zMz} and \eqref{3.zMz} that $M(\bm{c})$ 
is positive semidefinite and satisfies
$\bm{z}^T M(\bm{c})\bm{z}\le (\evmax-\lambda)|P_L(\bm{c})R(\bm{c})\bm{z}|^2$ for all
$\bm{z}\in\R^n$. Therefore, using Young's inequality with $\theta>0$,
\begin{align}\label{4.I10}
  I_{10} &\le \frac{\theta}{4}\sum_{i,j=1}^n\int_0^T\int_\Omega M_{ij}(\bm{c})
	\na\log\frac{c_i}{\bar{c}_i}\cdot\na\log\frac{c_j}{\bar{c}_j} dxdt \\
	&\phantom{xx}{}+ \frac{1}{\theta}\sum_{i,j=1}^n\int_0^T\int_\Omega M_{ij}(\bm{c})
	\na(\mu_i-\bar{\mu}_i)\cdot \na(\mu_j-\bar{\mu}_j)dxdt \nonumber \\
	&\le \frac{\theta}{4}(\evmax-\lambda)\sum_{i=1}^n\int_0^T\int_\Omega
	\bigg|\sum_{j=1}^n P_L(\bm{c})_{ij}\sqrt{c_j}\na\log\frac{c_i}{\bar{c}_i}\bigg|^2
	dxdt \nonumber \\
	&\phantom{xx}{}+ \frac{1}{\theta}(\evmax-\lambda)\sum_{i=1}^n\int_0^T\int_\Omega
	\bigg|\sum_{j=1}^n P_L(\bm{c})_{ij}\sqrt{c_j}\na(\mu_j-\bar{\mu}_j)\bigg|^2 dxdt.
	\nonumber
\end{align}

{\em Step 1b: Estimate of $I_{11}$.}
In the term $I_{11}$, we replace $\mu_j-\bar{\mu}_j$ by $\log(c_j/\bar{c}_j)
-\Delta(c_j-\bar{c}_j)$ and compute both terms in the difference separately.
The definition $G_{ij}(\bm{c})=\sqrt{c_i}P_L(\bm{c})_{ij}\sqrt{c_j}$ and
the property $P_L(\bm{c})^2=P_L(\bm{c})$ lead to
\begin{align}\label{4.G1}
  \sum_{i,j=1}^n&\int_0^T\int_\Omega G_{ij}(\bm{c})\na\log\frac{c_i}{\bar{c}_i}
	\cdot\na\log\frac{c_j}{\bar{c}_j}dxdt \\
	&= \sum_{i,j=1}^n\int_0^T\int_\Omega\sqrt{c_i} P_L(\bm{c})_{ij}\sqrt{c_j}
	\na\log\frac{c_i}{\bar{c}_i}\cdot\na\log\frac{c_j}{\bar{c}_j}dxdt \nonumber \\
	&= \sum_{i=1}^n\int_0^T\int_\Omega\bigg|\sum_{j=1}^n P_L(\bm{c})_{ij}
	\sqrt{c_j}\na\log\frac{c_j}{\bar{c}_j}\bigg|^2 dxdt. \nonumber
\end{align}
Furthermore, we use $G_{ij}(\bm{c}) = c_i\delta_{ij}-c_ic_j$ and integration by
parts to find that
\begin{align*}
	\sum_{i,j=1}^n&\int_0^T\int_\Omega
	G_{ij}(\bm{c})\na\log\frac{c_i}{\bar{c}_i}\cdot\na\Delta(c_j-\bar{c}_j)dx dt \\
	&= -\sum_{i,j=1}^n\int_0^T\int_\Omega\diver\bigg(( c_i\delta_{ij}-c_ic_j)
	\na\log\frac{c_i}{\bar{c}_i}\bigg)\Delta(c_j-\bar{c}_j)dxdt \\
	&= -\sum_{i=1}^n\int_0^T\int_\Omega\diver(\na c_i-c_i\na\log\bar{c}_i)
	\Delta(c_i-\bar{c}_i)dxdt \\
	&\phantom{xx}{}+ \sum_{i,j=1}^n\int_0^T\int_\Omega\diver(c_j\na c_i-c_ic_j
	\na\log\bar{c}_i)\Delta(c_j-\bar{c}_j)dxdt \\
	&= -\sum_{i,j=1}^n\int_0^T\int_\Omega\diver(\na c_i-c_i\na\log\bar{c}_i)
	\Delta(c_i-\bar{c}_i)dxdt \\
	&\phantom{xx}{}- \sum_{i,j=1}^n\int_0^T\int_\Omega\diver(c_ic_j
	\na\log\bar{c}_i)\Delta(c_j-\bar{c}_j)dxdt,
\end{align*}
where we used $\sum_{i=1}^n c_j\na c_i=0$ in the last step.
We mention that $\sum_{j=1}^n G_{ij}(\bm{c})\na\Delta c_j\in L^2(Q_T)$
because of \eqref{1.regc}, so the first integral in the previous computation
is well defined. It follows from $\Delta c_i\Delta(c_i-\bar{c}_i)
= (\Delta(c_i-\bar{c}_i))^2 + \Delta\bar{c}_i\Delta(c_i-\bar{c}_i)$ that
\begin{align}\label{4.G2}
  \sum_{i,j=1}^n&\int_0^T\int_\Omega G_{ij}(\bm{c})\na\log\frac{c_i}{\bar{c}_i}
	\cdot\na\Delta(c_i-\bar{c}_i) dxdt
	= -\sum_{i=1}^n\int_0^T\int_\Omega(\Delta(c_i-\bar{c}_i))^2 dxdt \\
	&\phantom{xx}{}
	- \sum_{i=1}^n\int_0^T\int_\Omega\diver(\na\bar{c}_i-c_i\na\log\bar{c}_i)
	\Delta(c_i-\bar{c}_i)dxdt \nonumber \\
	&\phantom{xx}{}- \sum_{i,j=1}^n\int_0^T\int_\Omega\diver(c_ic_j\na\log\bar{c}_i)
	\Delta(c_j-\bar{c}_j)dxdt. \nonumber
\end{align}
We multiply \eqref{4.G1} by $-\lambda$ and \eqref{4.G2} by $\lambda$ and
sum both expressions to find that
\begin{align}\label{4.I11}
  I_{11} &= -\lambda\sum_{i=1}^n\int_0^T\int_\Omega\bigg|\sum_{j=1}^n P_L(\bm{c})_{ij}
	\sqrt{c_j}\na\log\frac{c_j}{\bar{c}_j}\bigg|^2 dxdt
	- \lambda\sum_{i=1}^n\int_0^T\int_\Omega(\Delta(c_i-\bar{c}_i))^2 dxdt \\
	&\phantom{xx}{}- \lambda\sum_{i=1}^n\int_0^T\int_\Omega
	\diver(\na\bar{c}_i-c_i\na\log\bar{c}_i)\Delta(c_i-\bar{c}_i)dxdt \nonumber \\
	&\phantom{xx}{}- \lambda\sum_{i,j=1}^n\int_0^T\int_\Omega
	\diver(c_ic_j\na\log\bar{c}_i)\Delta(c_j-\bar{c}_j)dxdt. \nonumber
\end{align}
We apply Young's inequality to the last two terms. The third term in \eqref{4.I11}
becomes
\begin{align*}
  -&\lambda\sum_{i=1}^n\int_0^T\int_\Omega
	\diver(\na\bar{c}_i-c_i\na\log\bar{c}_i)\Delta(c_i-\bar{c}_i)dxdt \\
	&\le \frac{\lambda}{4}\sum_{i=1}^n\int_0^T\int_\Omega(\Delta(c_i-\bar{c}_i))^2 dxdt
	+ \lambda\sum_{i=1}^n\int_0^T\int_\Omega|\diver((c_i-\bar{c}_i)\na\log\bar{c}_i)|^2
	dxdt \\
	&\le \frac{\lambda}{4}\sum_{i=1}^n\int_0^T\int_\Omega(\Delta(c_i-\bar{c}_i))^2 
	dxdt \\
	&\phantom{xx}{}+ \lambda\sum_{i=1}^n\|\na\log\bar{c}_i\|_{L^\infty(Q_T)}
	\int_0^T\int_\Omega|\na(c_i-\bar{c}_i)|^2 dxdt \\
	&\phantom{xx}{}+ \lambda\sum_{i=1}^n
	\|\Delta\log\bar{c}_i\|_{L^\infty(Q_T)}
	\int_0^T\int_\Omega(c_i-\bar{c}_i)^2 dxdt \\
	&\le \frac{\lambda}{4}\sum_{i=1}^n\int_0^T\int_\Omega(\Delta(c_i-\bar{c}_i))^2 dxdt
	\\
	&\phantom{xx}{}+ \lambda C\sum_{i=1}^n\int_0^T\int_\Omega
	\big((c_i-\bar{c}_i)^2+|\na(c_i-\bar{c}_i)|^2\big)dxdt,
\end{align*}
where the constant $C>0$ depends on the $L^\infty$ norms of $\na\log\bar{\bm{c}}$
and $\Delta\log\bar{\bm{c}}$. Next, for the fourth term in \eqref{4.I11},
\begin{align*}
  -\lambda&\sum_{i,j=1}^n\int_0^T\int_\Omega
	\diver(c_ic_j\na\log\bar{c}_i)\Delta(c_j-\bar{c}_j)dxdt \\
	&\le \frac{\lambda}{4}\sum_{i=1}^n\int_0^T\int_\Omega(\Delta(c_i-\bar{c}_i))^2 dxdt
	+ \lambda\sum_{j=1}^n\int_0^T\int_\Omega\bigg|\sum_{i=1}^n\diver(c_ic_j
	\na\log\bar{c}_i)\bigg|^2 dxdt.
\end{align*}
We estimate the integrand of the last term, taking into account that
$\na\sum_{i=1}^n\bar{c}_i\na\log\bar{c}_i=\sum_{i=1}^n\na\bar{c}_i=0$:
\begin{align*}
  \sum_{i=1}^n&\diver(c_ic_j\na\log\bar{c}_i) = \sum_{i=1}^n\diver
	\big((c_i-\bar{c}_i)c_j\na\log\bar{c}_i\big) \\
	&= \sum_{i=1}^n c_j\diver((c_i-\bar{c}_i)\na\log\bar{c}_i)
	+ \sum_{i=1}^n(c_i-\bar{c}_i)\na\log\bar{c}_i\cdot\na c_j \\
	&= \sum_{i=1}^n c_j\diver((c_i-\bar{c}_i)\na\log\bar{c}_i)
	+ \sum_{i=1}^n c_i\na\log\bar{c}_i\cdot\na (c_j-\bar{c}_j) 
	+ \sum_{i=1}^n(c_i-\bar{c}_i)\na\log\bar{c}_i\cdot\na\bar{c}_j
	\\
	&\le C\sum_{i=1}^n\big(|c_i-\bar{c}_i| + |\na(c_i-\bar{c}_i)|\big),
\end{align*}
where $C>0$ depends on the $L^\infty$ norms of $\na\log\bar{\bm{c}}$ and
$\Delta\log\bar{\bm{c}}$. This yields
\begin{align*}
  -\lambda&\sum_{i,j=1}^n\int_0^T\int_\Omega
	\diver(c_ic_j\na\log\bar{c}_i)\Delta(c_j-\bar{c}_j)dxdt \\
	&\le \frac{\lambda}{4}\sum_{i=1}^n\int_0^T\int_\Omega(\Delta(c_i-\bar{c}_i))^2 dxdt
	+ \lambda C\sum_{i=1}^n\int_0^T\int_\Omega
	\big((c_i-\bar{c}_i)^2 + |\na(c_i-\bar{c}_i)|^2\big)dxdt.
\end{align*}
Using these estimates in \eqref{4.I11}, we arrive at
\begin{align}\label{4.I11final}
  I_{11} &\le -\lambda\sum_{i=1}^n\int_0^T\int_\Omega\bigg|\sum_{j=1}^n 
	P_L(\bm{c})_{ij}\sqrt{c_j}\na\log\frac{c_j}{\bar{c}_j}\bigg|^2 dxdt
	- \frac{\lambda}{2}\sum_{i=1}^n\int_0^T\int_\Omega(\Delta(c_i-\bar{c}_i))^2 dxdt \\
	&\phantom{xx}{}+ \lambda C\sum_{i=1}^n\int_0^T\int_\Omega
	\big((c_i-\bar{c}_i)^2 + |\na(c_i-\bar{c}_i)|^2\big)dxdt. \nonumber
\end{align}

{\em Step 1c: Estimate of $I_{12}$.}
By definition of $B_{ij}(\bm{c})$ and Young's inequality with $\theta'>0$, 
\begin{align*}
  I_{12} &= -\sum_{i,j=1}^n\int_0^T\int_\Omega\sqrt{c_i}\bigg(
	D_{ij}^{BD}(\bm{c})\sqrt{c_j} - \sqrt{\frac{c_i}{\bar{c}_i}}D_{ij}^{BD}(\bar{\bm{c}})
	\sqrt{\bar{c}_j}\bigg)\na\log\frac{c_i}{\bar{c}_i}\cdot\na\bar{\mu}_j dxdt \\
	&\le \frac{\theta'}{4}\sum_{i=1}^n\int_0^T\int_\Omega c_i\bigg|\na\log
	\frac{c_i}{\bar{c}_i}\bigg|^2 dxdt \\
	&\phantom{xx}{}+ \frac{n}{\theta'}\sum_{i,j=1}^n\int_0^T\int_\Omega\bigg(
	D_{ij}^{BD}(\bm{c})\sqrt{c_j} -  \sqrt{\frac{c_i}{\bar{c}_i}}
	D_{ij}^{BD}(\bar{\bm{c}})\sqrt{\bar{c}_j}\bigg)^2|\na\bar{\mu}_j|^2 dxdt.
\end{align*}
The bracket of the second term can be estimated according to
\begin{align}
  \bigg|D_{ij}^{BD}&(\bm{c})\sqrt{c_j} -  \sqrt{\frac{c_i}{\bar{c}_i}}
	D_{ij}^{BD}(\bar{\bm{c}})\sqrt{\bar{c}_j}\bigg| \label{4.estD} \\
	&= \bigg|D_{ij}^{BD}(\bm{c})\sqrt{c_j} - D_{ij}^{BD}(\bar{\bm{c}})\sqrt{\bar{c}_j}
	- \frac{\sqrt{c_i}-\sqrt{\bar{c}_i}}{\sqrt{\bar{c}_i}}D_{ij}^{BD}(\bar{\bm{c}})
	\sqrt{\bar{c}_j}\bigg| \nonumber \\
	&\le \frac{C}{\sqrt{m}}\sum_{i=1}^n\big(|c_i-\bar{c}_i| 
	+ |\sqrt{c_i}-\sqrt{\bar{c}_i}|\big)
	\le C(m)\sum_{i=1}^n|c_i-\bar{c}_i|, \nonumber
\end{align}
using the assumption $\bar{c}_i\ge m>0$ and the boundedness of $D_{ij}^{BD}$
(see Lemma \ref{lem.DB} (i)). It follows that
\begin{equation}
  I_{12} \le \frac{\theta'}{4}\sum_{i=1}^n\int_0^T\int_\Omega c_i\bigg|\na\log
	\frac{c_i}{\bar{c}_i}\bigg|^2 dxdt + C(m,\theta')\sum_{i=1}^n
	\int_0^T\int_\Omega(c_i-\bar{c}_i)^2 dxdt. \label{4.I12}
\end{equation}

{\em Step 1d: Combining the estimates.}
We deduce from \eqref{4.HGM}, after inserting estimates \eqref{4.I10}, 
\eqref{4.I11final}, and \eqref{4.I12} for $I_{10}$, $I_{11}$, and $I_{12}$,
respectively, that
\begin{align}
  \H(&\bm{c}(T)|\bar{\bm{c}}(T)) \le \H(\bm{c}^0|\bar{\bm{c}}^0) \nonumber \\
	&{}+ \bigg(\frac{\theta}{4}(\evmax-\lambda) - \lambda\bigg)\sum_{i=1}^n\int_0^T
	\int_\Omega\bigg|\sum_{j=1}^n P_L(\bm{c})_{ij}\sqrt{c_j}\na\log\frac{c_j}{\bar{c}_j}
	\bigg|^2 dxdt \label{4.I012} \\
	&{}+ \frac{\evmax-\lambda}{\theta}\sum_{i=1}^n\int_0^T\int_\Omega\bigg|
	\sum_{j=1}^n P_L(\bm{c})_{ij}\sqrt{c_j}\na(\mu_j-\bar{\mu}_j)\bigg|^2 dxdt 
	\nonumber \\
	&{}- \frac{\lambda}{2}\sum_{i=1}^n\int_0^T\int_\Omega(\Delta(c_i-\bar{c}_i))^2dxdt
	+ \lambda C\sum_{i=1}^n\int_0^T\int_\Omega\big((c_i-\bar{c}_i)^2
	+ |\na(c_i-\bar{c}_i)|^2\big)dxdt \nonumber \\
	&{}+ \frac{\theta'}{4}\sum_{i=1}^n\int_0^T\int_\Omega c_i\bigg|\na\log
	\frac{c_i}{\bar{c}_i}\bigg|^2 dxdt + C(m,\theta')\sum_{i=1}^n
	\int_0^T\int_\Omega(c_i-\bar{c}_i)^2 dxdt. \nonumber
\end{align}
The last but one term on the right-hand side still needs to be estimated.
To this end, we decompose $I=P_L(\bm{c})+P_{L^\perp}(\bm{c})$:
$$
  \sum_{i=1}^n c_i\bigg|\na\log\frac{c_i}{\bar{c}_i}\bigg|^2
	= \sum_{i=1}^n\bigg|\sum_{j=1}^n P_L(\bm{c})_{ij}\sqrt{c_j}
	\na\log\frac{c_j}{\bar{c}_j}\bigg|^2
	+ \sum_{i=1}^n\bigg|\sum_{j=1}^n P_{L^\perp}(\bm{c})_{ij}\sqrt{c_j}
	\na\log\frac{c_j}{\bar{c}_j}\bigg|^2.
$$
The first term on the right-hand side can be absorbed for sufficiently
small $\theta'>0$ by the second term of the left-hand side of \eqref{4.I012}.
For the other term, we use the definition $P_{L^\perp}(\bm{c})_{ij}=\sqrt{c_ic_j}$
and $\sum_{j=1}^n\na c_j=\sum_{j=1}^n\na\bar{c}_j=0$:
$$
  \sum_{j=1}^n P_{L^\perp}(\bm{c})_{ij}\sqrt{c_j}
	\na\log\frac{c_j}{\bar{c}_j} = \sqrt{c_i}\sum_{j=1}^n c_j
	\na\log\frac{c_j}{\bar{c}_j}
	= \sqrt{c_i}\sum_{j=1}^n (c_j-\bar{c}_j)\na\log\bar{c}_j.
$$
This gives
\begin{align}\label{4.clogcc}
  \sum_{i=1}^n\int_0^T\int_\Omega c_i\bigg|\na\log
	\frac{c_i}{\bar{c}_i}\bigg|^2 dxdt
	&\le \sum_{i=1}^n\int_0^T\int_\Omega\bigg|\sum_{j=1}^n P_L(\bm{c})_{ij}\sqrt{c_j}
	\na\log\frac{c_j}{\bar{c}_j}\bigg|^2 dxdt \\
	&\phantom{xx}{}+ \sum_{j=1}^n\|\na\log\bar{c}_j\|_{L^\infty(Q_T)}
	\int_0^T\int_\Omega(c_i-\bar{c}_i)^2 dxdt. \nonumber 
\end{align} 
Hence, choosing $\theta=\lambda/(\evmax-\lambda)$ and $\theta'=\lambda$, 
we conclude from \eqref{4.I012} that
\begin{align}\label{4.Hfinal}
  \H(&\bm{c}(T)|\bar{\bm{c}}(T)) + \frac{\lambda}{2}\sum_{i=1}^n\int_0^T
	\int_\Omega\bigg|\sum_{j=1}^n P_L(\bm{c})_{ij}\sqrt{c_j}\na\log\frac{c_j}{\bar{c}_j}
	\bigg|^2 dxdt \\
	&\phantom{xx}{}+ \frac{\lambda}{2}\sum_{i=1}^n\int_0^T\int_\Omega
	(\Delta(c_i-\bar{c}_i))^2dxdt \nonumber \\
	& \le \H(\bm{c}^0|\bar{\bm{c}}^0) + \frac{(\evmax-\lambda)^2}{\lambda}
	\sum_{i=1}^n\int_0^T\int_\Omega\bigg|
	\sum_{j=1}^n P_L(\bm{c})_{ij}\sqrt{c_j}\na(\mu_j-\bar{\mu}_j)\bigg|^2 dxdt 
	\nonumber \\
	&\phantom{xx}{}+ C\sum_{i=1}^n\int_0^T\int_\Omega\big((c_i-\bar{c}_i)^2
	+ |\na(c_i-\bar{c}_i)|^2\big)dxdt. \nonumber
\end{align}
We show in the next step that the second term on the right-hand side can be
estimated by the relative energy inequality.

{\em Step 2: Estimating the relative energy.} We start from the relative energy
inequality \eqref{4.relE}. Observing that due to Lemma \ref{lem.DB} (ii), 
\begin{align*}
  \sum_{i,j=1}^n B_{ij}(\bm{c})\na(\mu_i-\bar{\mu}_i)\cdot\na(\mu_j-\bar{\mu}_j)
	&= \sum_{i,j=1}^n D_{ij}^{BD}(\bm{c})\big(\sqrt{c_i}\na(\mu_i-\bar{\mu}_i)\big)
	\cdot\big(\sqrt{c_j}\na(\mu_j-\bar{\mu}_j)\big) \\
	&\ge \evmin\sum_{i=1}^n\bigg|\sum_{j=1}^n P_L(\bm{c})_{ij}\sqrt{c_j}
	\na(\mu_j-\bar{\mu}_j)\bigg|^2,
\end{align*}
inequality \eqref{4.relE} becomes
\begin{align}\label{4.aux2}
  & \E(\bm{c}(T)|\bar{\bm{c}}(T))
	+ \evmin\sum_{i=1}^n\int_0^T\int_\Omega\bigg|\sum_{j=1}^n P_L(\bm{c})_{ij}\sqrt{c_j}
	\na(\mu_j-\bar{\mu}_j)\bigg|^2dxdt \\
	&\phantom{xx}{}\le \E(\bm{c}^0|\bar{\bm{c}}^0) + I_{13} + I_{14} + I_{15} + I_{16},
	\qquad\mbox{where} \nonumber \\
	& I_{13} = -\sum_{i,j=1}^n\int_0^T\int_\Omega\bigg(B_{ij}(\bm{c})
	- \frac{c_i}{\bar{c}_i}B_{ij}(\bar{\bm{c}})\bigg)\na(\mu_i-\bar{\mu}_i)\cdot
	\na\bar{\mu}_j dxdt, \nonumber \\
	& I_{14} = \sum_{i,j=1}^n\int_0^T\int_\Omega(c_i-\bar{c}_i)\na(c_i-\bar{c}_i)
	\cdot\na\diver\bigg(\frac{1}{\bar{c}_i}B_{ij}(\bar{\bm{c}})\na\bar{\mu}_j\bigg)dxdt, 
	\nonumber \\
	& I_{15} = \frac12\sum_{i,j=1}^n\int_0^T\int_\Omega|\na(c_i-\bar{c}_i)|^2
	\diver\bigg(\frac{1}{\bar{c}_i}B_{ij}(\bar{\bm{c}})\na\bar{\mu}_j\bigg)dxdt, 
	\nonumber \\
	& I_{16} = \sum_{i,j=1}^n\int_0^T\int_\Omega
	\na(c_i-\bar{c}_i)\otimes\na(c_i-\bar{c}_i)
	:\na\bigg(\frac{1}{\bar{c}_i}B_{ij}(\bar{\bm{c}})\na\bar{\mu}_j\bigg)dxdt.
	\nonumber 
\end{align}
The terms $I_{14}$, $I_{15}$, and $I_{16}$ can be estimated directly by
using the regularity assumption 
$\na\diver((1/\bar{c}_i)B_{ij}(\bar{\bm{c}})\na\bar{\mu}_j)\in L^\infty(Q_T)$: 
\begin{equation}\label{4.I1456}
  I_{14}+I_{15}+I_{16} \le C\sum_{i=1}^n\int_0^T\int_\Omega\big(
  (c_i-\bar{c}_i)^2 + |\na(c_i-\bar{c}_i)|^2\big)dxdt.
\end{equation}

The estimate for $I_{13}$ is more involved. First, we use the definition of
$B(\bm{c})$ and decompose $I=P_L(\bm{c})+P_{L^\perp}(\bm{c})$. Then 
\begin{align*}
  & I_{13} = \sum_{i,j=1}^n\int_0^T\int_\Omega\sqrt{c_i} E_{ij}(\bm{c},\bar{\bm{c}})
	\na(\mu_i-\bar{\mu}_i)\cdot\na\bar{\mu}_j dxdt 
	=: I_{131} + I_{132}, \quad\mbox{where} \\
	& E_{ij}(\bm{c},\bar{\bm{c}}) = D_{ij}^{BD}(\bm{c})\sqrt{c_j}
	- \sqrt{\frac{c_i}{\bar{c}_i}}D_{ij}^{BD}(\bar{\bm{c}})\sqrt{\bar{c}_j}, \\
	& I_{131} = \sum_{i,j,k,\ell=1}^n\int_0^T\int_\Omega
	P_L(\bm{c})_{i\ell}E_{\ell j}(\bm{c},\bar{\bm{c}})
	P_L(\bm{c})_{ik}\sqrt{c_k}\na(\mu_k-\bar{\mu}_k)\cdot\na\bar{\mu}_j dxdt, \\
	& I_{132} = \sum_{i,j,k,\ell=1}^n\int_0^T\int_\Omega
	P_{L^\perp}(\bm{c})_{i\ell}E_{\ell j}(\bm{c},\bar{\bm{c}})
	P_{L^\perp}(\bm{c})_{ik}\sqrt{c_k}\na(\mu_k-\bar{\mu}_k)\cdot
	\na\bar{\mu}_j dxdt.
\end{align*}

For $I_{131}$, it is sufficient to apply Young's inequality and to use 
estimate \eqref{4.estD} for $E_{ij}(\bm{c},\bar{\bm{c}})$:
\begin{align}\label{4.I131}
  I_{131} &\le \frac{\evmin}{2}\sum_{i=1}^n\int_0^T\int_\Omega\bigg|\sum_{j=1}^n
	P_L(\bm{c})_{ij}\sqrt{c_j}\na(\mu_j-\bar{\mu}_j)\bigg|^2 dxdt \\
	&\phantom{xx}{}+ \frac{n}{2\evmin}\sum_{i,j=1}^n\int_0^T\int_\Omega
	|E_{ij}(\bm{c},\bar{\bm{c}})|^2|\na\bar{\mu}_j|^2 dxdt \nonumber \\
	&\le \frac{\evmin}{2}\sum_{i=1}^n\int_0^T\int_\Omega\bigg|\sum_{j=1}^n
	P_L(\bm{c})_{ij}\sqrt{c_j}\na(\mu_j-\bar{\mu}_j)\bigg|^2 dxdt \nonumber \\
	&\phantom{xx}{}+ C(m)\sum_{i=1}^n\int_0^T\int_\Omega(c_i-\bar{c}_i)^2 dxdt,
	\nonumber 
\end{align}
where $C(m)>0$ depends on $m$, $n$, $\evmin$, and the $L^\infty(Q_T)$ norm
of $\na\bar{\bm{\mu}}$. 

For $I_{132}$, we observe that the property $\ran D^{BD}(\bm{c})=L(\bm{c})$,
which follows from Lemma \ref{lem.Dz}, implies that
$P_{L^\perp}(\bm{c})D^{BD}(\bm{c})\bm{z}=\bm{0}$ for all $\bm{z}\in\R^n$.
Hence, 
$$
  \sum_{\ell=1}^n P_{L^\perp}(\bm{c})_{i\ell}E_{\ell j}(\bm{c},\bar{\bm{c}}) 
	= -\sum_{\ell=1}^n P_{L^\perp}(\bm{c})_{i\ell}\sqrt{\frac{c_\ell}{\bar{c}_\ell}}
	D_{\ell j}^{BD}(\bar{\bm{c}})\sqrt{\bar{c}_j}.
$$
We infer from the definitions
$P_{L^\perp}(\bm{c})_{ik}=\sqrt{c_ic_k}$ and
$\mu_k-\bar{\mu}_k=\log(c_k/\bar{c}_k)-\Delta(c_k-\bar{c}_k)$ that
\begin{align}\label{4.I132}
  I_{132} &= -\sum_{i,j,k,\ell=1}^n\int_0^T\int_\Omega P_{L^\perp}(\bm{c})_{ik}
	\sqrt{c_k}P_{L^\perp}(\bm{c})_{i\ell}\sqrt{\frac{c_\ell}{\bar{c}_\ell}}
	D_{\ell j}^{BD}(\bar{\bm{c}})\sqrt{\bar{c}_j}\na(\mu_k-\bar{\mu}_k)
	\cdot\na\bar{\mu}_j dxdt \\
	&= -\sum_{j,k,\ell=1}^n\int_0^T\int_\Omega\sum_{i=1}^n c_ic_k
	\frac{c_\ell}{\sqrt{\bar{c}_\ell}}D_{\ell j}^{BD}(\bar{\bm{c}})\sqrt{\bar{c}_j}
	\na(\mu_k-\bar{\mu}_k)\cdot\na\bar{\mu}_j dxdt \nonumber \\
	&= -\sum_{j,k,\ell=1}^n\int_0^T\int_\Omega c_k
	\frac{c_\ell-\bar{c}_\ell}{\sqrt{\bar{c}_\ell}}D_{\ell j}^{BD}(\bar{\bm{c}})
	\sqrt{\bar{c}_j}\na\log\frac{c_k}{\bar{c}_k}\cdot\na\bar{\mu}_j dxdt \nonumber \\
	&\phantom{xx}{}- \sum_{j,k,\ell=1}^n\int_0^T\int_\Omega \diver\bigg(c_k
	\frac{c_\ell-\bar{c}_\ell}{\sqrt{\bar{c}_\ell}}D_{\ell j}^{BD}(\bar{\bm{c}})
	\sqrt{\bar{c}_j}\na\bar{\mu}_j\bigg)\Delta(c_k-\bar{c}_k) dxdt \nonumber \\
	&=: J_1 + J_2, \nonumber 
\end{align}
where we added the expression $-\sum_{\ell=1}^n\sqrt{\bar{c}_\ell}
D_{\ell j}^{BD}(\bar{\bm{c}}) = 0$, which follows from
$\ker D^{BD}(\bar{\bm{c}})=L^\perp(\bar{\bm{c}})=\operatorname{span}
\{\sqrt{\bar{\bm{c}}}\}$ (see Lemma \ref{lem.DB}) and the symmetry of
$D^{BD}(\bar{\bm{c}})$ (see Lemma \ref{lem.Dz}), and we integrated by
parts in the last integral.

To estimate $J_1$, we use Young's inequality with $\theta>0$, 
Lemma \ref{lem.DB} (iii), and \eqref{4.clogcc}:
\begin{align*}
  J_1 &\le \frac{\theta}{4}\sum_{k=1}^n\int_0^T\int_\Omega c_k\bigg|
	\na\log\frac{c_k}{\bar{c}_k}\bigg|^2 dxdt \\
	&\phantom{xx}{}+ \frac{n}{\theta}\sum_{j,k,\ell=1}^n\int_0^T\int_\Omega
	(c_\ell-\bar{c}_\ell)^2\frac{c_k}{\bar{c}_\ell}D_{\ell j}^{BD}(\bar{\bm{c}})^2
	\bar{c}_j|\na\bar{\mu}_j|^2 dxdt \\
	&\le \frac{\theta}{4}\sum_{i=1}^n\int_0^T\int_\Omega\bigg|\sum_{j=1}^n
	P_L(\bm{c})_{ij}\sqrt{c_j}\na\log\frac{c_j}{\bar{c}_j}\bigg|^2 dxdt
	+ C\theta\sum_{i=1}^n\int_0^T\int_\Omega(c_i-\bar{c}_i)^2 dxdt \\
	&\phantom{xx}{}+ \frac{C}{\theta}\sum_{\ell=1}^n\int_0^T\int_\Omega
	(c_\ell-\bar{c}_\ell)^2 dxdt,
\end{align*}
where $C>0$ depends on the $L^\infty(Q_T)$ norms of $\na\bar{\bm{c}}$ and
$\na\bar{\bm{\mu}}$. 

Next, we use again Young's inequality with
$\theta'>0$:
\begin{equation}\label{4.J2}
  J_2 \le \frac{\theta'}{4}\sum_{k=1}^n\int_\Omega(\Delta(c_k-\bar{c}_k))^2dxdt
	+ \frac{n}{\theta'}\sum_{k,\ell=1}^n\int_0^T\int_\Omega
	\big|\diver\big(c_k(c_\ell-\bar{c}_\ell)Q_\ell(\bar{\bm{c}})\big)\big|^2 dxdt,
	\nonumber
\end{equation}
where we defined
$$
  Q_\ell(\bar{\bm{c}}) := \sum_{j=1}^n\frac{1}{\sqrt{\bar{c}_\ell}}
	D_{\ell j}^{BD}(\bar{\bm{c}})\sqrt{\bar{c}_j}\na\bar{\mu}_j.
$$
Estimating
\begin{align*}
  \big|\diver\big(c_k(c_\ell-\bar{c}_\ell)Q_\ell(\bar{\bm{c}})\big)\big|
	&= \big|c_k(c_\ell-\bar{c}_\ell)\diver Q_\ell(\bar{\bm{c}})
	+ c_k\na(c_\ell-\bar{c}_\ell)\cdot Q_\ell(\bar{\bm{c}}) \\
	&\phantom{xx}{}+ (c_\ell-\bar{c}_\ell)\na(c_k-\bar{c}_k)\cdot Q_\ell(\bar{\bm{c}})
	+ (c_\ell-\bar{c}_\ell)\na\bar{c}_k\cdot Q_\ell(\bar{\bm{c}})\big| \\
	&\le C\big(|c_\ell-\bar{c}_\ell| + |\na(c_\ell-\bar{c}_\ell)| 
	+ |\na(c_k-\bar{c}_k)|\big),
\end{align*}
where $C>0$ depends on the $L^\infty(Q_T)$ norm of $Q_\ell(\bar{\bm{c}})$,
we deduce from \eqref{4.J2} that
$$
  J_2 \le \frac{\theta'}{4}\sum_{k=1}^n\int_\Omega(\Delta(c_k-\bar{c}_k))^2dxdt
	+ \frac{C}{\theta'}\sum_{i=1}^n\int_0^T\int_\Omega\big((c_i-\bar{c}_i)^2
	+ |\na(c_i-\bar{c}_i)|^2\big)dxdt.
$$
Inserting the estimates for $J_1$ and $J_2$ into \eqref{4.I132} leads to
\begin{align*}
  I_{132} &\le \frac{\theta}{4}\sum_{i=1}^n\int_0^T\int_\Omega\bigg|\sum_{j=1}^n
	P_L(\bm{c})_{ij}\sqrt{c_j}\na\log\frac{c_j}{\bar{c}_j}\bigg|^2 dxdt
	+ \frac{\theta'}{4}\sum_{i=1}^n\int_0^T\int_\Omega(\Delta(c_i-\bar{c}_i))^2dxdt \\
	&\phantom{cc}{}+ C(\theta,\theta')\sum_{i=1}^n\int_0^T\int_\Omega
	\big((c_i-\bar{c}_i)^2 + |\na(c_i-\bar{c}_i)|^2\big)dxdt.
\end{align*}
Then, together with \eqref{4.I131}, we find that
\begin{align}
  I_{13} &\le \frac{\evmin}{2}\sum_{i=1}^n\int_0^T\int_\Omega\bigg|\sum_{j=1}^n
	P_L(\bm{c})_{ij}\sqrt{c_j}\na(\mu_j-\bar{\mu}_j)\bigg|^2 dxdt \nonumber \\
	&\phantom{xx}{}+ \frac{\theta}{4}\sum_{i=1}^n\int_0^T\int_\Omega\bigg|\sum_{j=1}^n
	P_L(\bm{c})_{ij}\sqrt{c_j}\na\log\frac{c_j}{\bar{c}_j}\bigg|^2 dxdt
	+ \frac{\theta'}{4}\sum_{i=1}^n\int_0^T\int_\Omega(\Delta(c_i-\bar{c}_i))^2dxdt 
	\label{4.I13} \\
	&\phantom{cc}{}+ C(\theta,\theta')\sum_{i=1}^n\int_0^T\int_\Omega
	\big((c_i-\bar{c}_i)^2 + |\na(c_i-\bar{c}_i)|^2\big)dxdt. \nonumber
\end{align}
Finally, we insert this estimate and estimate \eqref{4.I1456} for 
$I_{14}$, $I_{15}$, and $I_{16}$ into \eqref{4.aux2}, observing that the 
first term on the right-hand side of \eqref{4.I13} 
is absorbed by the second term on the left-hand side of \eqref{4.aux2}:
\begin{align}\label{4.Efinal}
  \E(&\bm{c}(T)|\bar{\bm{c}}(T)) + \frac{\evmin}{2}\sum_{i=1}^n
	\int_0^T\int_\Omega\bigg|\sum_{j=1}^n P_L(\bm{c})_{ij}\sqrt{c_j}
	\na(\mu_j-\bar{\mu}_j)\bigg|^2dxdt \\
  &\le \E(\bm{c}^0 |\bar{\bm{c}}^0) 
	+ \frac{\theta}{4}\sum_{i=1}^n\int_0^T\int_\Omega\bigg|\sum_{j=1}^n
	P_L(\bm{c})_{ij}\sqrt{c_j}\na\log\frac{c_j}{\bar{c}_j}\bigg|^2 dxdt \nonumber \\
	&\phantom{xx}{}
	+ \frac{\theta'}{4}\sum_{i=1}^n\int_0^T\int_\Omega(\Delta(c_i-\bar{c}_i))^2dxdt 
	\nonumber \\
	&\phantom{xx}{}+ C(\theta,\theta')\sum_{i=1}^n\int_0^T\int_\Omega
	\big((c_i-\bar{c}_i)^2 + |\na(c_i-\bar{c}_i)|^2\big)dxdt. \nonumber
\end{align}

{\em Step 3: Combining the relative energy and relative entropy inequalities.} 
Next, multiply \eqref{4.Efinal} by
$4(\evmax-\lambda)^2/(\evmin\lambda)$, choose
$\theta'=\lambda_m\lambda^2/(4(\lambda_M-\lambda)^2)$, and add this expression to 
\eqref{4.Hfinal} (which estimates $\H(\bm{c}|\bar{\bm{c}})$). 
Then some terms on the right-hand
side can be absorbed by the corresponding expressions on the left-hand side, 
leading to
\begin{align*}
  \H(&\bm{c}(T)|\bar{\bm{c}}(T)) + \frac{4(\evmax-\lambda)^2}{\evmin\lambda}
	\E(\bm{c}(T)|\bar{\bm{c}}(T)) \\
	&\phantom{xx}{}+ \frac{(\evmax-\lambda)^2}{\lambda}
	\sum_{i=1}^n\int_0^T\int_\Omega\bigg|\sum_{j=1}^n P_L(\bm{c})_{ij}\sqrt{c_j}
	\na(\mu_j-\bar{\mu}_j)\bigg|^2 dxdt \\
	&\phantom{xx}{}+ \frac{\lambda}{4}\sum_{i=1}^n\int_0^T\int_\Omega\bigg|\sum_{j=1}^n
	P_L(\bm{c})_{ij}\sqrt{c_j}\na\log\frac{c_j}{\bar{c}_j}\bigg|^2 dxdt
	+ \frac{\lambda}{4}\sum_{i=1}^n\int_0^T\int_\Omega(\Delta(c_i-\bar{c}_i))^2dxdt \\
	&\le \H(\bm{c}^0|\bar{\bm{c}}^0) + \frac{4(\evmax-\lambda)^2}{\evmin\lambda}
	\E(\bm{c}^0|\bar{\bm{c}}^0) \\
	&\phantom{xx}{}+ C(\theta,\theta')\sum_{i=1}^n\int_0^T\int_\Omega
	\big((c_i-\bar{c}_i)^2 + |\na(c_i-\bar{c}_i)|^2\big)dxdt.
\end{align*}
The last term can be bounded in terms of the free energy, since
$c_i\log(c_i/\bar{c}_i)-(c_i-\bar{c}_i)\ge (c_i-\bar{c}_i)^2/2$
\cite[Lemma 18]{HJT21}:
\begin{align*}
  \H(\bm{c}(T)|\bar{\bm{c}}(T)) + \frac{4(\evmax-\lambda)^2}{\evmin\lambda}
	\E(\bm{c}(T)|\bar{\bm{c}}(T)) 
	&\le \H(\bm{c}^0|\bar{\bm{c}}^0) + \frac{4(\evmax-\lambda)^2}{\evmin\lambda}
	\E(\bm{c}^0|\bar{\bm{c}}^0) \\
	&\phantom{xx}{}+ C\int_0^T\E(\bm{c}(t)|\bar{\bm{c}}(t))dt.
\end{align*}
Then the theorem follows after applying Gronwall's lemma.

%%%%%%%%%%%%%%%%%%%%%%%%%%%%%%%%%%%%%%%%%%%%%%%%%%%%%%%%%%%%%%%%%%%%%%%%%%%%%%

\section{Examples}\label{sec.exam}

We present some models which satisfy Assumptions (B1)--(B4).

\subsection{A phase separation model}

Elliott and Garcke have studied in \cite{ElGa97} equations
\eqref{1.eq1}--\eqref{1.mu}, formulated in terms of the mobility matrix
\eqref{1.B}, where
$$
  B_{ij}(\bm{c}) = b_i(c_i)\bigg(\delta_{ij} - \frac{b_j(c_j)}{\sum_{k=1}^n b_k(c_k)}
	\bigg), \quad i,j=1,\ldots,n.
$$
The functions $b_i\in C^1([0,1])$ are nonnegative and satisfy
$\beta_1 c_i\le b_i(c_i)\le \beta_2 c_i$ for $c_i\in[0,1]$ and some constants
$0<\beta_1\le \beta_2$. 
This model describes phase transitions in multicomponent systems;
it has been suggested in \cite{AkTo90} to model the dynamics of polymer mixtures 
with $b_i(c_i)=\beta_i c_i$ and $\beta_i>0$.
The subspace $L(\bm{c})$ becomes
$$
  L(\bm{c}) = \bigg\{\bm{z}\in\R^n:\sum_{i=1}^n \sqrt{b_i(c_i)}z_i=0\bigg\},
$$
and the matrix $D^{BD}(\bm{c})$ is determined directly from the mobility matrix:
$$
  D_{ij}^{BD}(\bm{c}) = \frac{B_{ij}(\bm{c})}{\sqrt{b_i(c_i)b_j(c_j)}}
	= \delta_{ij} - \frac{\sqrt{b_i(c_i)b_j(c_j)}}{\sum_{k=1}^n b_k(c_k)}.
$$
Instead of checking Assumptions (B1)--(B4), it is more convenient to verify
the statements of Lemma \ref{lem.DB} directly. This has been done in
\cite[Section 2]{HJT21}. Although the global existence of weak solutions
has been already proved in \cite{ElGa97}, we obtain the weak-strong uniqueness
property as a new result.

\subsection{Classical Maxwell--Stefan system}

In the classical Maxwell--Stefan model, the matrix $K(\bm{c})$ has the entries
$K_{ij}(\bm{c}) = \delta_{ij}\sum_{\ell=1}^n k_{i\ell}c_\ell - k_{ij}c_i$ for
$i,j=1,\ldots,n$. The associated matrix $D^{MS}(\bm{c})$ is given by
$$
  D_{ij}^{MS}(\bm{c}) = \frac{1}{\sqrt{c_i}}K_{ij}(\bm{c})\sqrt{c_j}
	= \delta_{ij}\sum_{\ell=1}^n k_{i\ell} c_\ell - k_{ij}\sqrt{c_ic_j}, 
	\quad i,j=1,\ldots,n.
$$
It is proved in \cite[Sec.~5.4]{HJT21} that this matrix satisfies
Assumptions (B1)--(B4). Thus, Theorems \ref{thm.ex} and \ref{thm.wsu} hold
for the model
\begin{align*}
  & \pa_t c_i + \diver(c_iu_i) = 0, \quad \sum_{i=1}^n c_iu_i = 0, \quad
	i=1,\ldots,n, \\
	& c_i\na\mu_i - \frac{c_i}{\sum_{k=1}^n c_k}\sum_{j=1}^n c_j\na\mu_j 
	= -\sum_{j=1}^n k_{ij}c_ic_j(u_i-u_j),
\end{align*}
where $\mu_i=\log c_i-\Delta c_i$.
Compared to \cite{HJT21}, the mobility does not only depend on $c_i$ but also
on $\Delta c_i$. This extends the existence and weak-strong uniqueness results
to a more general case.

\subsection{A physical vapor decomposition model for solar cells}

Thin-film crystalline solar cells can be fabricated as thin coatings on a 
substrate by the physical vapor decomposition process. The dynamics of the
volume fractions of the process components can be described by model
\eqref{1.eq1}--\eqref{1.bic} with the chemical potentials $\mu_i=\log c_i$
and the mobility matrix
$$
  B_{ij}(\bm{c}) = \delta_{ij}\sum_{\ell=1}^n k_{i\ell}  \, c_i c_\ell - k_{ij} c_ic_j,
	\quad i,j=1,\ldots,n.
$$
In this case, the Bott--Duffin matrix is given by 
$D_{ij}^{BD}(\bm{c}) = B_{ij}(\bm{c})/\sqrt{c_ic_j} = D_{ij}^{MS}(\bm{c})$,
where $D^{MS}(\bm{c})$ is the Maxwell--Stefan matrix of the previous subsection.
Thus, Assumptions (B1)--(B4) are verified for this matrix. We infer that
Theorems \ref{thm.ex} and \ref{thm.wsu} hold for the model
$$
  \pa_t c_i = \diver\sum_{j=1}^n k_{ij}c_ic_j\na(\mu_i-\mu_j), \quad
	\mu_i = \log c_i-\Delta c_i, \quad i=1,\ldots,n.
$$
When $\mu_i=\log c_i$ for all $i$, the global existence of weak solutions was proved in
\cite{BaEh18} and the weak-strong uniqueness of solutions was shown in
\cite{HoBu21}. A global existence result was obtained in \cite{EMP21} for
$\mu_1 = \log c_1 - \delta c_1 + \beta(1-2c_1)$ with $\beta>0$
and $\mu_i=\log c_i$ for $i=2,\ldots,n$. Our theorems extend these results 
to a more general case.

%%%%%%%%%%%%%%%%%%%%%%%%%%%%%%%%%%%%%%%%%%%%%%%%%%%%%%%%%%%%%%%%%%%%%%%%%%%%%%

\end{document}